\documentclass[11pt,twoside]{amsart}
\input{includeNice3}
\usepackage{mabliautoref}
\usepackage{amssymb}
\usepackage{amscd}
\usepackage[abbrev,alphabetic]{amsrefs}
\usepackage{hyperref}
\usepackage[a4paper,margin=1in]{geometry}  
\usepackage{CJKutf8}
\usepackage{soul}

\usepackage{xypic}
\usepackage{tikz-cd}
\usepackage{caption}

\usepackage{url}
\RequirePackage{booktabs, multirow}
\RequirePackage{pgf}
\title[Frobenius--stable GR vanishing fails for 3-folds]
{The Frobenius--stable version of the Grauert--Riemenschneider vanishing theorem fails} 
\author[J.~Baudin, F.~Bernasconi and T.~Kawakami]{Jefferson Baudin, Fabio Bernasconi and Tatsuro Kawakami} 

\address{\'Ecole Polytechnique F\'ed\'erale de Lausanne, SB MATH CAG, MA C3 615 (B\^atiment MA), Station 8, CH-1015 Lausanne, Switzerland.}
\email{jefferson.baudin@epfl.ch}

\address{Dipartimento di Matematica “Guido
Castelnuovo”, SAPIENZA Università di Roma, Piazzale Aldo Moro 5, I-00185 
Roma}
\email{fabio.bernasconi@uniroma1.it}

\address{Graduate School of Mathematical Sciences, University of Tokyo, 3-8-1 Komaba,
Meguro-ku, Tokyo 153-8914, Japan} 
\email{kawakami@ms.u-tokyo.ac.jp}

\DeclareMathOperator{\cent}{center}

\newcommand{\MO}{\mathcal{O}}
\newcommand{\sO}{\mathcal{O}}

\newcommand{\Q}{\mathbb{Q}}
\newcommand{\Z}{\mathbb{Z}}
\newcommand{\F}{\mathbb{F}}

\newcommand{\m}{\mathfrak{m}}

\begin{document}

\begin{abstract}
    We show that the Frobenius--stable version of the Grauert--Riemenschneider vanishing theorem fails for threefolds in any positive characteristic, and for terminal 3-folds in characteristic $p \in \{2, 3, 5\}$.
    To prove this, we introduce the notion of $\mathbb{F}_p$-rationality for singularities in positive characteristic and we show that klt singularities in dimension at most 4 are $\mathbb{F}_p$-rational.
    We apply this to prove a Frobenius--stable version of the Kawamata--Viehweg vanishing theorem on $K$-trivial 3-folds.
\end{abstract}

\maketitle

\tableofcontents

\section{Introduction}

The Grauert--Riemenschneider (GR) vanishing theorem states that, for a projective birational morphism $f \colon Y \to X$ of varieties defined over a field of characteristic 0 where $Y$ is regular, the higher direct images $R^if_*\omega_Y$ of the dualising sheaf vanish (\cite{GR70}, \cite{km-book}*{Corollary 2.68}).
It is one of the most celebrated consequences of the Kodaira vanishing theorem and it turned out to
be a useful tool to study singularities and deformations thereof in characteristic 0 (to name a
few, \cites{Elk78, EV85, Kov00}).

Let $p>0$ be a prime number. Over a field $k$ of characteristic $p$, GR vanishing always holds if $\dim(X) \leq 2$ by \cite{kk-singbook}*{Theorem 10.4} but it is well-known to fail starting from dimension at least 3. 
A typical counterexample is the cone over a smooth projective variety for which Kodaira vanishing fails (see \cite{HK15}*{Example 3.11}). 
Nevertheless, it is expected that some weakening of the GR vanishing theorem might still hold in characteristic $p$.
For example, a Witt vector version of GR vanishing stating that $R^if_*W\omega_{Y, \mathbb{Q}}$ vanishes has been conjectured in \cites{CR12, Tan22}.
%Various attempt to weaken the statement of GR vanishing and potentially obtain a true statement have been proposed in characteristic $p$, for example in terms of Witt vector cohomology (\cite{CR12, Tan22}), but its still open. 

In this article, we concentrate on the \emph{Frobenius--stable version of GR vanishing} and we show that such a statement can fail even for 3-fold singularities which naturally appear in the Minimal Model Program (MMP for short).

\begin{theorem}[\autoref{thm: counterex}] \label{thm: GR_fails_intro}
    For any $p > 0$, there exists a $\mathbb{Q}$-factorial affine 3-fold $X$ over $\mathbb{F}_p$ such that
    \[ R^1f_*\omega_Y \not\sim_C 0 \] for every resolution $f \colon Y \to X$.

    If furthermore $p \leq 5$, we can take $X$ to have only terminal singularities.
    % There exists a $\mathbb{Q}$-factorial $F$-pure terminal affine 3-fold $X$ over $\mathbb{F}_2$ such that, if $f \colon Y \to X$ denotes any resolution of singularities, then $R^1f_*\omega_Y \not\sim_C 0$.
\end{theorem}

\begin{remark}
\begin{enumerate}
    \item The bound $p \leq 5$ in the second part of the statement is optimal, because klt $3$-fold singularities always satisfy the usual Grauert--Riemenschneider vanishing when $p > 5$ \cite{BK23}.
    \item When $p = 2$, there even exists such an $X$ which is $F$-pure, and such that $K_X$ is Cartier (\autoref{rem: char_2_even_worse}). We do not know whether such examples exist for $p \in \{3, 5\}$ (see also \cite{Tot24}*{Remarks 6.2 and 8.2}).
\end{enumerate}
\end{remark}
Let us explain the shorthand $R^if_*\omega_Y \sim_C 0$.
The sheaf $\omega_Y$ of top differential forms on a smooth algebraic variety in positive characteristic carries an additional operation  first introduced by Cartier \cite{Cartier_une_nouvelle_operation_sur_les_formes_differentielles}: the Frobenius trace map
$
C\colon F_*\omega_Y \to \omega_Y$.
%Note that $C$ descends to the Cartier isomorphism (\cite{Cartier_une_nouvelle_operation_sur_les_formes_differentielles}) on the cohomology group $F_*\omega_Y/(F_*dF_*\Omega^{n-1}_Y)$. 
We write $R^if_*\omega_Y \sim_C 0$, and say that $X$ satisfies the Frobenius--stable GR vanishing, if for some $n>0$, the operator $(R^if_*C)^n$ 
vanishes for $i>0$.

Note that the cone counterexamples to the classical GR vanishing actually satisfy Frobenius--stable GR vanishing (see \cite{BBLSZ23}*{Proposition 5.18} or \autoref{ex: Frobenius-stable-GR-cone}). For this reason, we turn to another source of pathologies in characteristic $p$ to construct a counterexample: wild quotient singularities.

Indeed, Totaro's example of terminal 3-fold singularities $(x \ \in X)$ in characteristics $p \leq 5$ are obtained as (non-linear) quotients of $\mathbb{Z}/p\mathbb{Z}$ on a smooth 3-fold \cites{Tot19, Tot24} give the desired counterexamples.
To verify that $R^1f_*\omega_Y \not\sim_C 0$, instead of directly computing the operator $R^if_*C$ from a resolution $Y$, we reduce the non-vanishing in terms of the Frobenius action on local cohomology.
We first prove that the Frobenius action on $R^if_*\mathcal{O}_Y$ is nilpotent on 3-fold klt singularities (\autoref{thm: intro2}), which is the main technical result of this article (see \autoref{ss-fp-rat} for the details on these results). Having shown this, we can then apply the duality theory for Cartier crystals \cite{Bau23} to show the equivalence between the vanishing $R^1f_*\omega_Y \sim_C 0$ and the nilpotency of the Frobenius action on $H^2_x(\cO_{X,x})$.
Finally, some computations in group cohomology as in \cite{Forgarty_On_the_depth_of_local_rings_of_invariants_of_cyclic_groups} permit to conclude the proof when $p \leq 5$.

To obtain counterexamples for arbitrary $p > 0$, we also consider $\mathbb{Z}/p\mathbb{Z}$-quotients. Since these are not klt, we first compute by hand a partial resolution $\pi \colon Z \to X$, where $Z$ has toric singularities, and show directly that the Frobenius action on $R^i\pi_*\cO_Z$ is nilpotent for $i > 0$. 
Since $Z$ has toric singularities, we deduce that if $f \colon Y \to Z \to X$ is a resolution of singularities, then the Frobenius action on each $R^if_*\cO_Y$ is nilpotent for $i > 0$. We can then conclude that $Y \to X$ is a counterexample to GR vanishing up to nilpotence as above. \\

Let us compare our results with Bhatt vanishing for big and semi-ample line bundles in characteristic $p>0$ (see \cite{Bha12}*{Proposition 6.3}). 
Bhatt's theorem can be regarded as a Kawamata--Viehweg vanishing up to finite covers, meaning that non-zero cohomology classes are killed after some finite pull-back.
In general, we prove that the Frobenius morphism and its iterates are not sufficient to kill cohomology classes in the projective case (e.g. \autoref{cor: counter3}) and the failure of Frobenius--stable GR vanishing is a local analogue of this phenomenon.

\iffalse
\begin{remark}
    B. Totaro recently announced examples of terminal 3-fold singularities obtained as quotients of a cyclic group of order $p$ in characteristic $p \in \left\{ 3,5 \right\}$ (\cite{Tot23}). 
    Using these examples, one can extend the counterexample to Frobenius--stable GR vanishing to characteristic 3 and 5.
\end{remark}
\fi

\addtocontents{toc}{\SkipTocEntry}
\subsection{$\mathbb{F}_p$-rational singularities} \label{ss-fp-rat}

We found the counterexample to Frobenius--stable GR vanishing while looking for a weakening of the notion of rational singularities in positive characteristic that could help to understand klt singularities in positive characteristic. Let us explain our starting motivation.

The study of the singularities of the MMP played a crucial role in the development of higher dimensional algebraic geometry, both in characteristic 0 (see \cite{kk-singbook} and references therein) and $p>0$ (see \cites{Tan18, HX15, HW-MMp-5, HW22, HW19}).
This topic is closely intertwined with properties of Fano varieties given that klt singularities can be viewed as a local counterpart of Fano varieties \cite{LX16}*{Section 2.4}. 
In characteristic 0, a Fano variety satisfies the Kodaira vanishing theorem (which in particular implies that $H^i(X, \cO_X)=0$ for $i>0$) and a fundamental result of Elkik  \cite{Elk81} shows that klt singularities are Cohen--Macaulay and rational.
The situation is more complicated in characteristic $p>0$: Kodaira vanishing does not hold in general for Fano varieties \cites{Kov18, Tot19} and, furthermore, klt singularities may not be rational starting from dimension 3 \cites{CT19, GNT19, Tot19, Yas19, Ber21, ABL22}. 

From this perspective, we decided to explore whether a Frobenius--stable version of the Kodaira vanishing theorem holds for Fano varieties and investigate its consequence for cohomological properties of klt singularities. 
Our work was also driven by earlier findings in the literature concerning the $W\cO$-rationality and $\mathbb{Q}_p$-rationality of singular Fano varieties and klt singularities.
In \cites{GNT19, HW22}, it has been demonstrated that 3-fold klt singularities over perfect fields with characteristic $p>0$ are $W\cO$-rational, a weakening of rational singularities.
Moreover, in \cite{NT20} it is shown that the higher direct images $R^if_*W\cO_{X, \bQ}$ vanish for a Fano type morphism $f \colon X \to Z$ from a 3-fold $X$. Their result is then used in \cite{HW20} to conclude that 4-fold klt singularities over perfect fields are $W\cO$-rational.
These results immediately imply the analogue vanishing for the higher direct images $R^if_*\mathbb{Q}_p$.

We show that one can promote the vanishing in $\mathbb{Q}_p$-cohomology for klt singularities and Fano type varieties of \cite{NT20} to a vanishing of cohomology with $\mathbb{F}_p$-coefficients.
This can be interpreted via the Riemann--Hilbert (RH) correspondence for $\mathbb{F}_p$-constructible sheaves \cites{Bockle_Pink_Cohomological_Theory_of_crystals_over_function_fields, Bhatt_Lurie_RH_corr_pos_char}: asking the topological vanishing $R^if_*\bF_p=0$ for $i>0$ is equivalent to requesting the natural action of the Frobenius on $R^if_*\cO_X$ to be nilpotent for $i>0$ (denoted by $R^if_*\cO_X \sim_F 0$). 
Taking this perspective, our result on the cohomology of Fano varieties is the following:

\begin{theorem}[cf.~\autoref{prop: nilp_vanishing_dP} and \autoref{thm: vanishing_3fold_Fanotype}] \label{thm: intro_1}
    Let $f \colon X \to Z$ be a morphism of Fano type over an $F$-finite field $k$.
    \begin{enumerate}
        \item\label{inrto1-(1)} If $\dim(X)=2$, then $R^if_*\cO_X \sim_F 0$ for $i>0$;
        \item\label{inrto1-(2)} If $\dim(X)=3$, $k$ is perfect and $p>3$, then $ R^if_*\cO_X \sim_F 0$ for $i>0$.
    \end{enumerate}
\end{theorem}
\begin{remark}
    \begin{enumerate}
        \item In \autoref{thm: intro_1}.\autoref{inrto1-(1)}, we cannot expect $R^if_*\cO_X=0$ in general at least in characteristic $p=2$ or $3$ over imperfect fields (\cites{Maddock,  Sch07} and \cite{Tan20}*{Theorem 1.4}). 
        Nevertheless, we know that if $\dim\,f(X)\geq 1$ (resp.~$Z=\Spec(k)$ and $p>5$) then $R^if_*\cO_X=0$ by \cite{Tan18}*{Theorem 3.3} (resp.~\cite{BT22}*{Theorem 1.7}). We do not know if there exists a surface of del Pezzo type over an imperfect field of characteristic $5$ with non-zero irregularity.
        \item In \autoref{thm: intro_1}.\autoref{inrto1-(2)}, the assumption that $p>3$ is used only for the (absolute) MMP for 3-folds.
              In particular, for a $\Q$-factorial klt Fano 3-fold $X$ over a perfect field of characteristic $p$ with Picard number 1, we can show that $H^i(X, \sO_X)\sim_{F} 0$ for all $i>0$ and all $p>0$ (see \autoref{thm:klt Fano}.\autoref{klt Fano (a)}). 
              It is important to note that if $p=2$, then $H^2(X,\sO_X)= 0$ does not hold in general \cite{CT19}*{Theorem 1.4}, and it is open whether $H^i(X,\sO_X)=0$  for all $i>0$ holds for large $p\gg 0$.
              
              In \autoref{thm: intro_1}.\autoref{inrto1-(2)}, we know that if $p>5$ and $\dim\,f(X)\geq 1$ then $R^if_*\cO_X=0$ by \cite{BK23}*{Theorem 5.1}. Moreover, if $X$ is a smooth Fano 3-fold we know that $H^i(X, \cO_X)=0$ for $i>0$ holds by \cite{SB97} (cf.~\cite{Kaw21}*{Corollary 3.7} or \cite{FanoI}*{Theorem 2.4}). 
    \end{enumerate}
\end{remark}

\begin{remark}
    A similar principle, where taking powers of the Frobenius fixes the possible pathologies for contractions of Fano type, appears in the $p$-power version of the base point free theorem proven by Tanaka \cites{tan-p-power}.
\end{remark}

Following the global-to-local principle, we are to able to deduce, using the birational MMP up to dimension 4 and \autoref{thm: intro_1}, the $\mathbb{F}_p$-rationality of klt singularities up to dimension $4$.

\begin{theorem}[cf.~\autoref{prop: 3-fold_sing} and \autoref{cor: 4folds_klt}]\label{thm: intro2}
    Let $(R, \mathfrak{m})$ be a local $k$-algebra essentially of finite type such that $X=\Spec(R)$ is of klt type. 
    Let $f \colon Y \to X$ be a projective resolution of singularities.
    \begin{enumerate}
        \item\label{inrto2-(1)} If $\dim(X)=3$, then $ R^if_*\cO_Y\sim_F 0$ for $i>0$.
        \item\label{inrto2-(2)} Assume that $k$ is perfect and the existence of log resolution of singularities for 4-folds over $k$.
        If $\dim(X)=4$ and $p>5$, then $R^if_*\cO_Y \sim_F 0$ for $i>0$.
    \end{enumerate}
\end{theorem}

\begin{remark}
   \begin{enumerate}
       \item \autoref{thm: intro2}.\autoref{inrto2-(1)} is really a novelty in the case of perfect field of characteristic $p \leq 5$ and imperfect fields, as klt 3-fold singularities are rational over perfect fields of characteristic $p>5$ \cites{HW19, ABL22, BK23}.
       It is worth mentioning that, when $p\leq 5$, we cannot expect $ R^if_*\cO_Y=0$ for $i>0$ in general even if $k$ is algebraically closed \cites{CT19, Ber21, ABL22}.
    \item In \autoref{thm: intro2}.\autoref{inrto2-(2)}, the assumption that $p>5$ is used only for the MMP in dimension at most 4 developed in \cite{HW20}.
    To the best of our knowledge, it is open whether there exists a bound $p_0\in\Z_{>0}$ for klt 4-folds such that $R^if_{*}\sO_Y=0$ in characteristic $p>p_0$. 
   \end{enumerate}
\end{remark}

We conclude the article with an application to a new vanishing theorem for $K$-trivial 3-folds in characteristic $p>0$. 
We show that a version of the Kawamata--Viehweg vanishing theorem holds on $K$-trivial 3-folds up to the action of the Frobenius.

\begin{theorem}[cf.~ \autoref{thm: kvv-up-frob}]
    Let $X$ be a projective 3-fold over an $F$-finite field with $K_X \sim_{\mathbb{Q}} 0$ and klt singularities.
    If $L$ is big and semiample line bundle, then for $e \gg 0$, the morphism $F^e \colon H^i(X, L) \to H^i(X, L^{\otimes p^e}) $ is zero for $e \gg 0$ and $i > 0$.
    
    If $X$ is moreover assumed to be globally $F$-split, then $H^i(X, L)=0$ for all $i > 0$.
    \end{theorem}

We find it interesting that for $K$-trivial 3-folds the Frobenius morphism is enough to kill cohomology classes of a big and semiample line bundle, although this is not the case in general (see \autoref{cor: counter3}). \\

Note that our techniques also give partial vanishing results for $K$-trivial 4-folds, see \autoref{rem: vanishing_CY_fourfolds}.

\addtocontents{toc}{\SkipTocEntry}
\subsection{Organisation of the article}
We briefly explain the organisation of the article. 
In \autoref{s-prelims}, we recall the notions of Frobenius crystals and Cartier crystals, which give the correct framework to study properties that are stable under the action of the Frobenius. 
We also include the tools from the MMP that are relevant for our purposes.
In \autoref{s-def}, we introduce the Frobenius--stable version of the notion of singularities we are interested in: Cohen--Macaulay up to Frobenius nilpotence, $\mathbb{F}_p$-rationality and the Frobenius--stable version of GR vanishing theorem. 
We relate the Frobenius action on local cohomology of an $\mathbb{F}_p$-rational singularity to the Cartier structure on the higher direct images of the dualising sheaf (\autoref{lem: equiv_conj_GR_CM}) and we compare $\mathbb{F}_p$-rationality with the notions of $\mathbb{Q}_p$ and $W\cO$-rationality (\autoref{ss-comparison}).
In \autoref{s: global-local}, we implement the global--to--local principle: the MMP provides a tool to show the $\mathbb{F}_p$-rationality of klt singularities from the vanishing up to Frobenius nilpotence of the cohomology of Fano varieties (\autoref{thm: F_p-rationality-Fano}).
We then show \autoref{thm: intro_1}.\autoref{inrto1-(1)} and we deduce $\mathbb{F}_p$-rationality of klt singularities in dimension 3. 
From this, we are able to obtain \autoref{thm: GR_fails_intro}.
Finally in \autoref{s: vanishing3foflds} we show the nilpotence of Frobenius on 3-folds of Fano type over perfect fields (\autoref{thm:threefold}) building on their rational chain-connectedness and some computations with $p$-adic cohomology theories. 
Using the MMP developed in \cite{GNT19} and \cite{HW22}, we show \autoref{thm: intro_1}.\autoref{inrto1-(2)} and \autoref{thm: intro2}.\autoref{inrto2-(2)}.
In the last section, we show how to apply the previous results to obtain a Frobenius--stable vanishing theorem on klt Calabi--Yau 3-folds.

\addtocontents{toc}{\SkipTocEntry}
\subsection*{Acknowledgements}
We would like to thank Zs.~Patakfalvi, C.~D.~Hacon, S.~Filipazzi, K.~Schwede,  T.~Takamatsu, B.~Totaro, and the referees for useful discussions and comments on the content of this article. 

JB was partly supported by the grant \#200020B/192035 from the Swiss National Science Foundation, FB was partly supported by the grants \#200020B/192035 and PZ00P2-216108 from the Swiss National Science Foundation and TK was partly supported by JSPS KAKENHI Grant numbers JP22KJ1771 and JP24K16897, and by the Inamori Foundation.
Part of this work was done while the second author was visiting Kyoto University, and he would like to express his gratitude to TK for the warm hospitality.

\section{Preliminaries} \label{s-prelims}

\subsection{Notations}

\begin{enumerate}
    \item Throughout the article, we fix a prime number $p>0$.
    \item For an $\mathbb{F}_p$-scheme $X$, we denote by $F \colon X \to X$ its (absolute) Frobenius morphism, and we use the same notation for $\mathbb{F}_p$-algebras. We say that an $\mathbb{F}_p$-scheme $X$ (resp.~algebra) is $F$-finite if the Frobenius is a finite morphism. 
    Whenever we say $F$-finite scheme (or algebra), we intend an $F$-finite $\mathbb{F}_p$-scheme. 
    \item On an $\bF_p$-scheme $X$, we define the perfection of $\mathcal{O}_X$ as 
    \[ \cO_X^{1/p^{\infty}} \coloneqq \varinjlim \: (\cO_X \xrightarrow{F} F_*\cO_X \xrightarrow{F} F^2_*\cO_X \xrightarrow{F} \dots) \] (see \cite{Bhatt_Lurie_RH_corr_pos_char}). We use the same notation for $\bF_p$-algebras.
    \item Given a scheme $X$, we denote by $X_{\red}$ its reduced subscheme. If $X$ is integral, we denote by $\nu \colon X^{\nu} \to X$ its normalisation and by $k(X)$ its function field.
    \item We say $X$ is a variety over $k$ (or simply a variety) if it is an integral separated scheme of finite type over $k$.
    \item We say $X$ is a curve (resp.~surface, 3-fold, 4-fold) over $k$ if $X$ is a variety over $k$ of dimension one (resp.~two, three, four). 
    \item Given an abelian category $\mathcal{A}$, we denote by $D(\mathcal{A})$ (resp.~$D^b(\mathcal{A})$) its (bounded) derived category.
    Given a complex $K^{\bullet}$ in $D(\mathcal{A})$ and $i \in \bZ$, we denote by $\mathcal{H}^i(K^{\bullet})$ its $i$-th cohomology.
    \item Given a variety $X$ over $k$ with structural morphism $\pi \colon X \to\Spec k$, we define the canonical dualising complex $\omega_X^{\bullet} \coloneqq \pi^!\cO_{\Spec k}$ (see \stacksprojs{0A9Y}{0ATZ}). 
    The canonical dualising sheaf is defined as the lowest non-zero cohomology sheaf of the dualising complex: 
    $\omega_X \coloneqq \mathcal{H}^{-\dim X}(\omega_X^{\bullet})$. For example, if $X$ is proper and smooth over $k$, then $\omega_X$ agrees with the usual sheaf of top differentials \stacksproj{0AU3}.
    If $X$ is normal, then $\omega_X$ is a reflexive sheaf of rank 1 \stacksproj{0AWE}, and we denote by $K_X$ a Weil divisor such that $\omega_X \cong \cO_X(K_X)$. Whenever we work with varieties over $k$, we will implicitly work with the canonical dualising complexes and sheaves described above.
    \item In general, if $\omega_X^{\bullet}$ is a dualising complex on a Noetherian integral scheme $X$, then we set $\omega_X \coloneqq \mathcal{H}^j(\omega_X^{\bullet})$, where $j$ is the minumum of all integers $i$ such that $\mathcal{H}^i(\omega_X^{\bullet}) \neq 0$.
    \item We say $(X,\Delta)$ is a \emph{pair} if $X$ is a normal variety, $\Delta$ is an effective $\mathbb{Q}$-divisor and $K_X+\Delta$ is $\mathbb{Q}$-Cartier. 
    We say $\Delta$ is a boundary if its coefficients lie in the interval $[0,1]$. 
    \item Given a variety $X$, we say that $f \colon Y \to X$ is a \emph{resolution} if $f$ is a projective birational morphism such that $Y$ is regular. If in addition the exceptional locus $\Ex(f)$ is purely divisorial and snc, then we say that $f$ is a \emph{log resolution}.
    \item Given a variety $X$, we say that $f \colon Y \to X$ is a \emph{quasi-resolution} if $f$ is a projective and generically finite morphism, $Y$ is a topological finite quotient \cite[Definition 4.2.5]{CR12}, and the extension of the function fields $k(Y) \subset k(X)$ is purely inseparable.
\end{enumerate}

\subsection{Frobenius and Cartier modules} \label{ss: F-Cartier-mod}

In this section, we recall the definitions and the basic results on the theory of Frobenius modules and Cartier modules we need, which are contained in \cites{Bockle_Pink_Cohomological_Theory_of_crystals_over_function_fields, BB11, BS13, Bhatt_Lurie_RH_corr_pos_char, Bau23}.

\subsubsection{Frobenius modules}

We start with defining Frobenius modules and recalling  the Riemann--Hilbert correspondence in positive characteristic \cites{Bockle_Pink_Cohomological_Theory_of_crystals_over_function_fields, Bhatt_Lurie_RH_corr_pos_char}. Since all the schemes we will be working with are varieties, we will mostly restrict to the Noetherian case. Most of what we say can be generalised to the non-Noetherian setup by \cite{Bhatt_Lurie_RH_corr_pos_char}.

\begin{definition} \label{def: F_mod}
Let $X$ be an $\mathbb{F}_p$-scheme and let $\mathcal{M}$ be an $\mathcal{O}_X$-module.
A \emph{Frobenius-module structure} on $\mathcal{M}$ is an $\mathcal{O}_X$-module homomorphism $\tau \colon \mathcal{M} \to F_* \mathcal{M}$.
We say $(\mathcal{M}, \tau)$ is an \emph{Frobenius module}. If $\mathcal{M}$ is quasi-coherent (resp.~coherent), we say $(\mathcal{M}, \tau)$ is a quasi-coherent (resp.~coherent) Frobenius module.
We say that $f \colon (\mathcal{M}_1, \tau_1) \to (\mathcal{M}_2, \tau_2)$ is a \emph{homomorphism of Frobenius modules} if $ F_*f \circ \tau_1 = \tau_2 \circ f$.  
\end{definition}

We denote by $\QCoh_X^F$ (resp.~$\Coh_X^F$ if $X$ is Noetherian) the category of quasi-coherent (resp.~coherent) Frobenius modules. The category $\QCoh_X^F$ is a Grothendieck category by \cite{Bau23}*{Corollary 2.2.4}.
We illustrate the key examples of Frobenius modules that appear in this article.

\begin{example} \label{ex: F-mod-struct}
For any $\mathbb{F}_p$-scheme $X$, the structure sheaf $\mathcal{O}_X$ comes with the absolute Frobenius $F \colon \mathcal{O}_X \to F_*\mathcal{O}_X$, which induces a natural Frobenius module structure on $\cO_X$. 
Similarly, any ideal sheaf $\mathcal{I} \subset \mathcal{O}_X$ canonically admits a Frobenius module structure.
From this, we can construct new Frobenius module structures.

\begin{enumerate}
    \item Let $f \colon X \to Y$ be a morphism of $\mathbb{F}_p$-schemes. Then the action of the Frobenius $F \colon \mathcal{O}_X \to F_*\mathcal{O}_X$ induces a natural Frobenius module structure on the higher direct images:
    $$R^if_*F \colon R^if_*\mathcal{O}_X \to R^if_*F_* \mathcal{O}_X \cong F_*R^if_* \mathcal{O}_X.$$
    \item Let $(X, x)=(\Spec(R), \mathfrak{m})$, where $R$ is a local $\mathbb{F}_p$-algebra. 
    Then the local cohomology groups $H^i_x(\mathcal{O}_X):=H^i_\mathfrak{m}(R)$ admit a natural structure of a quasi-coherent Frobenius modules.
\end{enumerate}
\end{example}

Given a Frobenius module $(\mathcal{M}, \tau)$ we define the composite 

$$ \begin{tikzcd}
\tau^n \colon \cM \arrow[r, "\tau"] & F^{n -1}_*\cM \arrow[r, "F_*\tau"] & \dots \arrow[r, "F^{n - 1}_*\tau"] & F^n_*\cM.
\end{tikzcd} $$

\begin{definition}
    We say that a Frobenius module $(\mathcal{M}, \tau)$ is \emph{nilpotent} if there exists $n>0$ such that $\tau^n=0$, and that it is \emph{locally nilpotent} if it is a union of nilpotent Frobenius submodules.
\end{definition}

Locally nilpotent Frobenius modules form a Serre subcategory of Frobenius modules by \cite{Bockle_Pink_Cohomological_Theory_of_crystals_over_function_fields}*{Proposition 3.3.5}. 
The category of Frobenius quasi-crystals (resp.~Frobenius crystals) $\QCrys_X^F$ (resp.~$\Crys_X^F$) is the Serre quotient of $\QCoh_X^F$ (resp.~$\Coh_X^F$) by the subcategory of locally nilpotent submodules (see \stacksproj{02MS}). Note that $\QCrys_X^F$ is a Grothendieck category by \cite{Bau23}*{Lemma 2.2.11}.
With an abuse of notation, we will often simply write $\mathcal{M}$ instead of the couple $(\mathcal{M}, \tau)$.

\begin{definition}
Any two quasi-coherent Frobenius modules $\mathcal{M}_1$ and $\mathcal{M}_2$ are \emph{isomorphic up to nilpotence} if the associated Frobenius quasi-crystals are isomorphic. In this case, we write $\mathcal{M}_1 \sim_F \mathcal{M}_2$.
\end{definition}

\begin{lemma}\label{universal_homeomorphism_induces_isomorphism_of_crystals}
    Let $f \colon X \to Y$ be a universal homeomorphism of Noetherian $\bF_p$-schemes. Then the natural morphism $\mathcal{O}_Y \to Rf_*\cO_X$ is an isomorphism in $D(\Crys_X^F)$. 
    %$f_*\cO_X \sim_F \cO_Y$ and $R^if_{*}\sO_X=0$ for all $i>0$. \textcolor{red}{Maybe, we should state $Rf_*\cO_X \sim_F \cO_Y$ since we use this lemma for the equality $H^i(\sO_X)\sim_{F} H^i(\sO_Y)$ later.}\FB{I agree!} 
\end{lemma}
\begin{proof}
    Since $f$ is affine by \stacksproj{04DE}, it suffices to show that $\cO_Y \sim_F f_*\cO_X$.
    We show that the natural morphism $\cO_Y^{1/p^{\infty}} \to f_*\cO_X^{1/p^{\infty}}$ is an isomorphism. Since $f$ is affine, we may assume that both $X$ and $Y$ are affine. The statement then follows from \stacksproj{0CNF}.
\end{proof}

A fundamental theorem in the theory of Frobenius crystals in the \emph{Riemann-Hilbert correspondence} (in short, RH correspondence), which relates Frobenius crystals to constructible $\mathbb{F}_p$-sheaves. 

\begin{definition}[\cite{Bockle_Pink_Cohomological_Theory_of_crystals_over_function_fields}*{10.1.3} and \textup{\cite{Bhatt_Lurie_RH_corr_pos_char}*{Construction 10.2.2}}]
    Given a Frobenius module $(\cM, \tau)$ on an $\bF_p$-scheme $X$, we will define an étale sheaf of $\bF_p$-vector spaces on $X$. Let $U \to X$ be an étale morphism. 
    Pulling back $(\mathcal{M}, \tau)$ gives a Frobenius module $(\cM_U, \tau_U)$ on $U$. We then set \[ \Sol(\cM)(U) \coloneqq \ker(\tau_U - 1 \colon \cM_U \to \cM_U). \]
    This defines a functor $\Sol \colon \QCoh_X^F \to \Sh(X_{\et}, \bF_p)$, where $\Sh(X_{\et}, \bF_p)$ denotes the abelian category of étale sheaves of $\bF_p$-vector spaces on $X$.
\end{definition}

Note that a nilpotent Frobenius module cannot have any fixed points, so when $X$ is Noetherian, $\Sol$ factors through $\Crys_X^F \to \Sh(X_{\et}, \bF_p)$.

\begin{theorem}[{\cite{Bockle_Pink_Cohomological_Theory_of_crystals_over_function_fields}*{Theorems 10.3.6 and 10.4.2}}]\label{thm:RH}
    Let $X$ be a Noetherian $\bF_p$-scheme. Then $\Sol$ induces an equivalence of categories 
    \[ \Crys_X^F \xrightarrow{\cong} \Sh_c(X_{\et}, \bF_p), \]
     where $\Sh_c(X_{\et}, \bF_p)$ denotes the full subcategory of $\Sh(X_{\et}, \bF_p)$ of constructible \'etale sheaves of $\bF_p$-vector spaces.
    Furthermore, for any morphism $f \colon Y \to X$ of Noetherian $F$-finite schemes, the functor satisfies the following compatibility properties:
    \begin{enumerate}
        \item $\Sol (f^*\mathcal{M})\cong f^*\Sol(\mathcal{M})$ for any $\mathcal{M} \in \Crys_X^F$;
        \item if $f$ is proper, then $\Sol (R^if_*\mathcal{N})\cong R^if_*\Sol(\mathcal{N})$ for any $\mathcal{N} \in \Crys_Y^F$ and $i \geq 0$.
    \end{enumerate}
\end{theorem}

\begin{example}[see also \textup{\cite{Bhatt_Lurie_RH_corr_pos_char}*{Theorem 10.5.5}}] \label{ex: RH}
    Let $f \colon Y \to X$ be a proper morphism of Noetherian $\mathbb{F}_p$-schemes.
    In this case, the Frobenius crystals $R^if_*\mathcal{O}_Y$ correspond to $R^if_*\mathbb{F}_{p,Y}$ for every $i \geq 0$ via the RH correspondence as it commutes with higher proper pushforwards.
\end{example}

\begin{lemma}\label{lem: vanishing fibers}
    Let $X$ be a Noetherian $\bF_p$-scheme, and let $\cM \in \Crys_X^F$. Then \[ \cM \sim_F 0 \iff \cM \otimes k(x) \sim_F 0 \: \: \mbox{ for all } x \in X. \]
\end{lemma}

\begin{proof}
    The ``left-to-right'' direction is immediate. To see the other direction, note that by \autoref{thm:RH}, it is sufficient to prove that $\Sol(\mathcal{M})=0$, which is equivalent to $\Supp(\Sol( \mathcal{M}))=\emptyset$. Hence, we conclude by \cite{Bau23}*{Lemma 2.2.17}.
\end{proof}

We recall the proper base change theorem for Frobenius crystals.

\begin{theorem}\label{proper_base_change_F_crystals}
    Let $f \colon Y \to X$ be a proper morphism of Noetherian $\bF_p$-schemes, and let $\cM$ be a Frobenius crystal on $Y$. Then for any point $x \in X$, we have \[ R^if_*\cM \otimes k(x) \sim_F H^i(Y_x, \cM_{Y_x}), \]
    where $Y_x$ is the scheme-theoretic fibre over $x$.
\end{theorem}
\begin{proof}
    This is \cite{Bockle_Pink_Cohomological_Theory_of_crystals_over_function_fields}*{Theorem 6.7.6} applied to the inclusion $\Spec k(x) \to X$.
\end{proof}

\subsubsection{Cartier modules}

\begin{definition} \label{def: Cart_mod}
Let $X$ be a Noetherian $\bF_p$-scheme, and let $\mathcal{N}$ be an $\mathcal{O}_X$-module.
A \emph{Cartier module structure} on $\mathcal{N}$ is an $\mathcal{O}_X$-module homomorphism $\kappa \colon F_* \mathcal{N} \to \mathcal{N}$.
We say that $(\mathcal{N}, \kappa)$ is a \emph{Cartier module}. 
If $\mathcal{N}$ is quasi-coherent (resp.~coherent), we say that $(\mathcal{N}, \kappa)$ is a quasi-coherent (resp.~coherent) Cartier module.
\end{definition}

Quasi-coherent (resp.~coherent) Cartier modules form a Grothendieck category (resp.~an abelian category) by \cite{Bau23}*{Corollary 3.1.3}, denoted by $\QCoh_X^{C}$ (resp $\Coh_X^C$).
The main sources of examples of Cartier modules come from duality theory.

\begin{example}
Let $X$ be a normal variety over an $F$-finite field $k$.
Let $\mathcal{O}_{X^{\reg}} \to F_*\cO_{X^{\reg}}$ be the Frobenius morphism and we apply the functor $\mathcal{H}om(-, \omega_{X^{\reg}}).$
Applying Grothendieck duality, we obtain \[\mathcal{H}om(F_*\mathcal{O}_{X^{\reg}}, \omega_{X^{\reg}}) \cong F_*\mathcal{H}om(\mathcal{O}_{X^{\reg}}, \omega_{X^{\reg}}) \cong F_*\omega_{X^{\reg}}.\]
As $X$ is normal and the sheaves are all reflexive, we conclude $\mathcal{H}om(F_*\mathcal{O}_{X}, \omega_{X}) \cong F_*\omega_X$ and thus we have the Frobenius trace map:
$$C \colon F_*\omega_X \to \omega_X, $$
which endows $\omega_X$ with a natural structure of Cartier module.

%This can be generalised to couples $(X,D)$, where $X$ is normal and $D$ is an effective integral Weil divisor. In this case, we obtain the composition \[ \begin{tikzcd} F_*\omega_X(D) \arrow[r] &  F_*\omega_X(pD) \arrow[rr, "\Tr \otimes \cO(D)"] &  & \omega_X(D). \end{tikzcd} \] We denote this morphism by $\Tr_{(X, D)}$ and it gives a natural Cartier module structure on $\omega_X(D)$.
\end{example}

Given a Cartier module $(\mathcal{N}, \kappa)$ we define the composite 

$$ \begin{tikzcd}
\kappa^n\colon F^n_*\cN \arrow[r, "F^{n - 1}_*\kappa"] & F^{n -1}_*\cN \arrow[r, "F^{n - 2}_*\kappa"] & \dots \arrow[r, "\kappa"] & \cN.
\end{tikzcd} $$

\begin{definition} \label{def: Cartier crystal}
We say that a Cartier module $(\mathcal{N}, \kappa)$ is \emph{nilpotent} if there exists $n>0$ such that $\kappa^n=0$, and that it is \emph{locally nilpotent} if it is a union of nilpotent Cartier submodules.
\end{definition}

Similarly to Frobenius modules, locally nilpotent Cartier modules form a Serre subcategory \cite{Bau23}*{Lemma 3.3.3}.
The category $\Crys_X^C$ of Cartier crystals is the Serre quotient category of $\Coh_X^C$ by $\Nil_X^C$. \\

In \cite{Bau23}, the first author develops a duality theory between Frobenius crystals and Cartier crystals, which we now recall. We refer to \stacksproj{0A9Y} for the definition of upper-shriek functors.

\begin{definition}
    A \emph{unit dualising complex} on a Noetherian $F$-finite scheme $X$ is the datum of a dualising complex $\omega_X^{\bullet}$, together with an isomorphism $\omega_X^{\bullet} \to F^!\omega_X^{\bullet}$. This agrees with \cite{Bau23}*{Definition 4.2.1} by Propositions 5.1.1 and 5.1.3 in \emph{loc.~cit.}
\end{definition}

By \cite{Bau23}*{Corollary 5.1.15}, unit dualising complexes always exist for schemes of finite type over a Noetherian $F$-finite affine scheme.

\begin{theorem}[\cite{Bau23}*{Theorem 4.3.5 and Corollary 5.1.7}]\label{thm: duality}
    Let $X$ be a Noetherian $F$-finite scheme with a unit dualising complex $\omega_X^{\bullet}$. Then the usual duality functor $D_X(-) \coloneqq \RHHom(-, \omega_X^{\bullet})$ (see \stacksproj{0A89}) can be enriched to an equivalence of categories \[ D_X \colon D^b(\Crys_X^C)^{op} \xrightarrow{\cong} D^b(\Crys_X^F). \]
    
    Moreover, for any proper morphism $f \colon Y \to X$ from a Noetherian $F$-finite scheme $Y$, if we set $\omega_Y^{\bullet} = f^!\omega_X^{\bullet}$, then we have a natural isomorphism
        \[ D_X \circ Rf_* \cong Rf_* \circ D_Y. \] 
  \iffalse
    Let $f \colon Y \to X$ be a proper morphism of Noetherian $F$-finite schemes. 
    Suppose that $X$ has a unit dualising complex $\omega_X^{\bullet}$ and let $\omega_Y^\bullet$ the induced unit dualising complex on $Y$ (\cite{Bau23}*{Corollary 5.1.4}).
    Let $(A, B) \in \left\{(F, C), (C, F )\right\}$.
    Then there is a canonical isomorphism as functors $D^b(\Crys_Y^A)^{op} \to D^b(\Crys_X^B)$:
    \[ \mathbf{R}f_* \mathbf{R}\mathcal{H}om_{Y}(-, \omega_Y^{\bullet}) \cong \mathbf{R}\mathcal{H}om_X(\mathbf{R}f_*(-), \omega_X^{\bullet}). \]
    \fi
\end{theorem}

\iffalse
\begin{remark}
    In \autoref{thm: duality}, we pointed out that the $Rf_*$ is taken in the category $\text{IndCrys}$. It turns out that this higher pushforward agrees with the usual one on $\cO_X$-modules in the following sense: for any $A \in \{F, C\}$, the following diagram commutes

    \[  \begin{tikzcd}
        D^+(\cO_Y) \arrow[rr, "Rf_*"]                           &  & D^+(\cO_X)                           \\
        D^+(\IndCoh^A_Y) \arrow[rr, "Rf_*"] \arrow[d] \arrow[u] &  & D^+(\IndCoh^A_X) \arrow[d] \arrow[u] \\
        D^+(\IndCrys_Y^A) \arrow[rr, "Rf_*"]                    &  & D^+(\IndCrys_X^A)                   
    \end{tikzcd} \]

    (see \cite{Bau23}*{Proposition 2.2.12 and }
\end{remark}

\fi

\subsection{Notions from the Minimal Model Program}

Let $X$ be a normal scheme, and let $\Delta$ be an effective $\bQ$-divisor with coefficients $\leq 1$ such that $K_X + \Delta$ is $\bQ$-Cartier. We recall the definition of singularities for the pair $(X, \Delta)$ according to the MMP.
Given a proper birational morphism $f \colon Y \to X$ with $Y$ normal, we choose a canonical divisor $K_Y$ such that $f_*K_Y=K_X$, and we write
$$ K_Y + f_*^{-1}\Delta = f^*(K_X+\Delta)+\sum_i a(E_i, X, \Delta) E_i, $$
where $E_i$ run through the prime exceptional divisors and $f_*^{-1}\Delta $ is the strict transform of $\Delta$. 
The coefficient $a(E, X, \Delta)$ is called the \emph{discrepancy} of $E$ (with respect to $(X,\Delta)$), and it depends exclusively on the divisorial valutation associated to $E$.

\begin{definition}
    We say that $(X, \Delta)$ is klt if $\lfloor \Delta \rfloor =0$ and $a(E, X, \Delta) >-1$. We say that $X$ is \emph{of klt type} if there exists $\Delta \geq 0$ such that $(X,\Delta)$ is klt.
    
    We say $(X, \Delta)$ is log canonical if $a(E, X, \Delta) \geq -1$.
\end{definition}

For inductive reasons in our proofs, the class of dlt singularities appear naturally when running an MMP for snc pairs with a reduced boundary.
Given a log canonical pair $(X, \Delta)$, we say that a divisor $E$ on a birational model of $X$ is a \emph{log canonical place} if it is a divisorial valuation of discrepancy -1.

\begin{definition}
    We say that $(X, \Delta)$ is \emph{dlt} if $(X, \Delta)$ is log canonical, and for every log canonical place $E$, we have that the generic point of $\cent_X(E)$ is contained in the snc locus of $(X, \Delta)$.
\end{definition} 

For a discussion of klt and log canonical singularities including the case of positive characteristic, we refer to \cite{kk-singbook} and \cite{7authors}.
We now recall the notion of Fano type morphism, which generalises the notion of Fano variety to pairs in the relative setting.

\begin{definition}
    Let $f \colon X \to Y$ be a projective morphism of varieties.
    We say $f$ is a \emph{morphism of Fano type} if there exists $\Delta \geq 0$ such that $(X, \Delta)$ is klt and $-(K_X+\Delta)$ is $f$-ample.
\end{definition}

\subsection{Non-$\bQ$-factorial 4-fold MMP}

In this section we explain how to adapt the arguments in \cite{HW22, HW20} to prove the existence of a birational MMP over non-$\mathbb{Q}$-factorial bases as in \cite{k-notQfact} for 4-folds, assuming the existence of log resolutions. 

\begin{hypothesis}
\label{hyp}
Following \cites{HW20}, we assume that, given an integral 4-fold $X$, log resolutions of all log pairs with the underlying variety being birational to $X$ exist and are given by a sequence of blowups along the non-snc locus.
\end{hypothesis}

The following is a generalisation of \cite{HW20}*{Theorem 1.1} to non-$\mathbb{Q}$-factorial bases.

\begin{theorem} \label{thm: kollar-notqfactorial-4folds}
    Let $X$ be a 4-fold satisfying \autoref{hyp} defined over a perfect field $k$ of characteristic $p>5$.
    Let $\pi \colon Y \to X$ be a projective birational morphism with exceptional divisor $E=\sum E_i$.
    Suppose there exist $\mathbb{R}$-divisors $\Delta$ and $H$ on $Y$ such that
    \begin{enumerate}
        \item[(i)] \label{non_Q_fact i} $(Y, \Delta)$ is dlt, $\lfloor{ \Delta \rfloor} = \Supp(E)$ and each $E_i$ is $\bQ$-Cartier;
        \item[(ii)] $K_Y +\Delta \equiv_{\pi} \sum e_i E_i $;
        \item[(iii)] $H = \sum h_i E_i$, where $-H$ is effective and $\Supp H = E$. 
        \item[(iv)] $K_Y +\Delta+r_X H$ is $\pi$-ample for some $r_X>0$.
        \item[(v)] The $h_i$ are linearly independent over $\mathbb{Q}(e_1 , \dots , e_n )$.
    \end{enumerate}

Then we can run the $(K_Y+\Delta)$-MMP with scaling of $H$ over $X$ as in \cite{k-notQfact}*{Definition 1} which terminates with $\pi_{\min} \colon Y_{\min} \to X$.
Assume that we reached the $j$-th step of this MMP: 
\begin{center}
    \begin{tikzcd}
        Y_j \arrow[rr, dotted] \arrow[dr,"\varphi_j",swap] & & Y_{j+1} \arrow[dl,"\psi_{j}"] \\
             & Z_j. & \\
    \end{tikzcd}
\end{center}
Let $E^j$ be the divisorial part of the exceptional locus of the birational contraction $\pi_j \colon Y_j \to X$ and $\Delta_{Y_j}$ (resp. $H^j$) be the strict transform of $\Delta$ (resp. $H$) on $Y_j$.
Let $$r_j \coloneqq \inf \left\{ t \geq 0 \mid K_{Y_j}+\Delta_{Y_j}+tH^j \text{ is nef}\right\} $$ be the threshold of the MMP with scaling of $H$ at the $j$-th step. 
We have the following properties:
\begin{enumerate}
    \item \label{non_Q_fact a} given two irreducible components $E^j_{l_1}$ and $E^j_{l_2}$ of $E^j$, there exists $\lambda=\lambda(l_1,l_2)$ such that $E^j_{l_1} \equiv \lambda E^j_{l_2}$ in $N^1(Y_j/Z_j)$ (in particular, they are all multiples of $K_{Y_j}+\Delta_{Y_j}$).
    \item \label{non_Q_fact b} $\Ex(\varphi_j) \subset \Supp(E^j)$ and either
    \begin{enumerate}
        \item $\Ex(\varphi_j)$ is an irreducible divisor and $Y_{j+1} = Z_{j}$, or
        \item  $\varphi_j$ is small, and there are irreducible components $E^{j}_{j_1}$ and $E^{j}_{j_2}$ of $E^j$ such that $E^{j}_{j_1}$ and $-E^{j}_{j_2}$ are both $\varphi_{j}$-ample.
    \end{enumerate}
    \item \label{non_Q_fact c} The irreducible components of $E^{j+1}$ are all $\bQ$-Cartier;
    \item \label{non_Q_fact d} $\sum e_i E_i^{j+1} +(r_j - \varepsilon)H^{j+1}$ is a $\pi_{j+1}$-ample $\mathbb{R}$-divisor supported on $\Ex(\pi_{j+1})$ for $0 < \varepsilon \ll 1$. 
    \end{enumerate}
Furthermore, 
if $\sum e_iE_i$ is effective and $\Supp (\sum e_iE_i) =E$, then $Y_{\min}=X$.
\end{theorem}

\begin{proof}
    The proof of \cite{k-notQfact}*{Theorem 2} for \autoref{non_Q_fact b}, \autoref{non_Q_fact c} and \autoref{non_Q_fact d} (see also \cite{7authors}*{Theorem 9.15}) works for 4-folds, as soon as we prove the existence of the desired contractions and flips. 
    The explanation for \autoref{non_Q_fact a} can be found at lines 5-8 of \cite{7authors} on page 195.
    
    In our setting, the cone theorem is valid by the same proof of \cite{HW20}*{Theorem 4.6}. 
    Pick an extremal face $F$ as in the proof of \cite{k-notQfact}*{Theorem 1.1}.
    The same proof as in \cite{HW20}*{Proposition 4.5} shows that there exists a birational contraction which contracts exactly the face $F$. 
    Thus we are left to construct the desired flips.
    We suppose first that $\Delta$ has standard coefficients.
    In this case, using the $\bQ$-factoriality of the irreducible prime divisor in $\Ex(\pi_j)$, we can reduce to the case where $\Delta=E_{j_1} +E_{j_2} +\Gamma$ such that $\lfloor \Gamma \rfloor=0$ and $\Gamma$ has standard coefficients. 
    Then the required flip exists by the same argument as in \cite{HW20}*{Theorem 4.3}.
    If $\Delta$ does not have standard coefficients, we use \autoref{hyp} to reduce to this case as done in \cite{HW20}*{Theorem 4.1}.

    As for the termination of the sequence of flips, we have that the exceptional locus of each step of the MMP is contained in $E^j$ and thus the MMP must terminate by special termination \cite{XX22}*{Theorem 3.4} (in their proof one can replace the $\mathbb{Q}$-factoriality of $Y$ with Hypothesis (i)).%\autoref{item: qcartier}).
\end{proof}

\begin{corollary} \label{cor: 4-fold-klt-base}
    Let $(X, \Delta)$ be a 4-fold pair satisfying \autoref{hyp} over a perfect field $k$ of characteristic $p>5$.
    Let $\pi \colon Y \to X$ be a log resolution, whose exceptional locus $E$ supports a divisor that is ample over $X$.
    If $(X,\Delta)$ is klt, then we can run a $(K_Y+\pi_*^{-1}\Delta+E)$-MMP in the sense of \autoref{thm: kollar-notqfactorial-4folds} which terminates with $X$.
\end{corollary}

\begin{proof}
    We have that 
    $(K_Y+\pi_*^{-1}\Delta+E) \equiv_\pi \sum_i (1+a(E_i,X,\Delta))E_i,$ which is effective by the klt hypothesis.
    As $E$ supports an ample divisor, we can find an $\mathbb{R}$-divisor $H$ satisfying hypotheses (iii), (iv) and (v) of \autoref{thm: kollar-notqfactorial-4folds}, so we conclude. 
\end{proof}

\section{Frobenius--stable versions of CM, rationality and GR vanishing} \label{s-def}

In this section, we define and discuss basic properties of Frobenius--stable versions of Cohen--Macaulay modules, rational singularities and GR vanishing. Apart from examples, we will allow ourselves to work with general Noetherian and $F$--finite schemes, not only varieties.

\subsection{Definitions}

We introduce the analogue of Serre's conditions for Frobenius crystals and Cartier crystals.
Similar notions of depth for local cohomology up to Frobenius action have been investigated in previous works (see for example \cites{HS77, Lyu06}).

\begin{definition}
    Let $X$ be a Noetherian $F$-finite scheme, and let $A \in \{F, C\}$.
    \begin{itemize}
        \item Given $\cM^{\bullet} \in D^b(\Crys_X^A)$, define \[ \Supp_{\crys}(\cM^{\bullet}) := \{x \in X\mid \: \cM^{\bullet}_x \not\sim_A 0 \}. \]
        This subset is closed by \cite{Bau23}*{Lemma 2.2.17}. We also define \[ \dim_{\crys}(\cM^{\bullet}) \coloneqq \dim(\Supp_{\crys}(\cM^{\bullet})).\]
        \item Let $(R, \fm)$ be a Noetherian $F$-finite local ring, and let $0 \not\sim_A M \in \Crys_X^A$. We set \[ \depth_{\crys}(M) \coloneqq \min\{i \geq 0\mid \: H^i_{\fm}(M) \not\sim_A 0\}. \]
        We say that $M$ has \emph{Serre's condition $(S_n)$ up to nilpotence} if \[\depth_{\crys}(M) \geq \min\{n, \dim_{\crys}(M) \}.\] If $M$ has $(S_n)$ up to nilpotence for all $n \geq 0$, then $M$ is said to be \emph{Cohen-Macaulay up to nilpotence} (in short, CM up to nilpotence). We set the zero crystal to be CM up to nilpotence by definition.
        \item Let $\cM \in \Crys_X^A$. We say that $\cM$ has property $(S_n)$ up to nilpotence (resp.~$\cM$ is CM up to nilpotence) if, for all $x \in X$, the localised sheaf $\cM_x$ has this property.
    \end{itemize}
\end{definition}

Thanks to the following lemma, this notion of depth up to nilpotence behaves very similarly to the usual notion of depth.

\begin{lemma}\label{lem: equivalent_def_CM_up_to_nilp}
   Let $(A, B) \in \{(F, C), (C, F)\}$.
   Let $(R, \fm)$ be a Noetherian $F$-finite local ring, fix a normalised unit dualising complex $\omega_R^{\bullet}$, and let $0 \not\sim_A M \in \Crys_R^A$. Then $\depth_{\crys}(M) \leq \dim_{\crys}(M)$, and \[ \Ext^{-i}(M, \omega_R^{\bullet}) \not\sim_B 0 \implies \: i \in \{\depth_{\crys}(M), \dots, \dim_{\crys}(M)\}, \] and $\Ext^{-\delta}(M, \omega_R^{\bullet}) \not\sim_B 0$ for $\delta \in \{\depth_{\crys}(M), \dim_{\crys}(M) \}$.

   In particular, $M$ is CM up to nilpotence if and only if $\mathrm{RHom}(M, \omega_R^{\bullet})$ is supported in a single degree as a $B$-crystal.
\end{lemma}
\begin{proof}
    By \cite{Bau23}*{Proposition 4.4.5}, we have a local duality theorem for crystals: for all $i \in \bZ$, \[ \Ext^{-i}(M, \omega_R^{\bullet}) \sim_B 0 \iff H^i_{\fm}(M) \sim_A 0. \] Throughout the proof, we will refer to this result as local duality.
    
    Let $Z \coloneqq \Supp_{\crys}(M)$, let $i \colon Z \to \Spec R$ denote the corresponding closed immersion, and write $Z = \Spec S$. By \cite{Bau23}*{Lemma 2.2.17 and Corollary 5.3.5}, there exists $N \in \Coh_Z^A$ such that $i_*N \sim_A M$. In particular, for all $i < -\dim(Z)$, \[ \Ext^i(M, \omega_R^{\bullet}) \sim_B \Ext^i(i_*N, \omega_R^{\bullet}) \sim_B 0 \] by \cite{stacks-project}*{Tag \href{https://stacks.math.columbia.edu/tag/0A7U}{0A7U}}. We then obtain from local duality that if $\Ext^{-i}(M, \omega_R^{\bullet}) \not\sim_B 0$, then $i \leq \dim_{\crys}(M)$. The other bound is immediate from the definition (and again local duality). Let us now show that $\Ext^{-j}(M, \omega_R^{\bullet}) \not\sim_B 0$ for $j = \dim(Z)$ (in particular, this implies that $\depth_{\crys}(M) \leq \dim_{\crys}(M)$ by local duality). Let $\eta$ be a generic point of some irreducible component of $Z$ (in particular, $N_{\eta} \not\sim_A 0$). Note that \[ \Ext_R^{-j}(M, \omega_R^{\bullet})_{\eta} \sim_B \cHH^{-j}\left(\RHom_R(i_*N, \omega_R^{\bullet})\right)_{\eta} \expl{\sim_{B}}{\autoref{thm: duality}} \cHH^{-j}\left(\RHom_S(N, \omega_S^{\bullet})\right)_{\eta} \sim_B \Ext^{-j}_{k(\eta)}(N_{\eta}, (\omega_S^{\bullet})_{\eta}).  \] Since $(\omega_S^{\bullet})_{\eta} = \omega_{k(\eta)}^{\bullet}[j]$ , we conclude that \[ \Ext_R^{-j}(M, \omega_R^{\bullet})_{\eta} \sim_B \Hom(N_{\eta}, \omega_{k(\eta)}^{\bullet}). \] Since $N_{\eta} \not\sim_A 0$, also $\Hom(N_{\eta}, \omega_{k(\eta)}^{\bullet}) \not\sim_B 0$ ($k(\eta)$ is a field) so the proof that $\Ext^{-j}(M, \omega_X^{\bullet}) \not\sim_A 0$ is complete. 
    
    Finally, the non-vanishing for $j = \depth_{\crys}(M)$ is immediate from the definition and local duality.
\end{proof}

\begin{definition}
    Let $X$ be a Noetherian $F$-finite $\mathbb{F}_p$-scheme. We say that $X$ has \emph{property $(S_n)$ up to nilpotence} (resp.~$X$ is \emph{Cohen--Macaulay (CM) up to nilpotence}) if the Frobenius crystal $\cO_X$ has property $(S_n)$ up to nilpotence (resp.~is CM up to nilpotence).
\end{definition}

We define rational and strongly rational singularities up to Frobenius nilpotence by enlarging the definition of rational singularities to the category of crystals.

\begin{definition} \label{def: nilp_rat_sing}
Let $f \colon Y \to X$ be a projective morphism of Noetherian $F$-finite schemes. 
We say the morphism $f$ is $\mathbb{F}_p$\emph{-rational} if $\cO_X \to Rf_*\cO_Y$ is an isomorphism in $D^b(\Crys_X^F)$.

We say that an integral $\bF_p$-scheme $X$ has \emph{$\mathbb{F}_p$-rational singularities} if there exists an $\mathbb{F}_p$-rational projective resolution $f \colon Y \to X$ of singularities.
We say that $X$ has \emph{strongly $\mathbb{F}_p$-rational singularities} if $X$ has $\bF_p$-rational singularities, and is CM up to nilpotence.
\end{definition}

\begin{remark}\label{rem: enough to check vanishing of the exceptional}
We have the following properties.
\begin{enumerate}
    \item By \autoref{ex: RH}, we have the following equivalences 
    \begin{align*}
    &f_*\mathcal{O}_Y \sim_F \mathcal{O}_X \iff f_*\mathbb{F}_{p,Y} \cong \mathbb{F}_{p,X},\,\,\text{and}\\
    &R^if_*\cO_Y \sim_F 0 \iff R^if_*\bF_{p, Y} = 0 \text{ for } i>0,
    \end{align*}
    where $\bF_{p, Y}$ is seen as an étale sheaf on $Y$. 
    This explains the name of $\mathbb{F}_p$-rationality.
     \item If $f \colon Y \to X$ is a resolution, then $X$ has $\bF_p$-rational singularities if and only if $R^if_*\cO_E \sim_F 0$ for $i > 0$, where $E$ denotes the exceptional locus of $f$.
    This is an application of the proper base change theorem:
    For any point $x \in f(E)$, we consider the following diagram
    \begin{center}
        \begin{tikzcd}
            f^{-1}(x) \arrow[r] \arrow[d] & E \ar[dr] \ar[r] &  Y \arrow[d] \\
             x  \arrow[rr] &  & X, \\
        \end{tikzcd}
    \end{center}
    and, applying \autoref{proper_base_change_F_crystals} twice, we deduce that \[R^if_* \cO_Y \otimes k(x) \sim_F H^i(f^{-1}(x), \cO_{f^{-1}(x)}) \sim_F R^if_* \cO_E \otimes k(x).\] 
    We then conclude since, by \autoref{lem: vanishing fibers}, $R^if_*\cO_E \sim_F 0$ if and only if for all $x \in X$, $R^if_*\cO_E \otimes k(x) \sim_F 0$.
    % (\textcolor{red}{T: I don't know why this is a consequence of Thm 2.9. Is there a point such that $E=Y_{x}$? If $E$ is the exceptional divisor (not locus) and $f$ is projective (Btw, do we assume resolution is projective?), then we can deduce it from $0\to \sO_Y(-E)\to \sO_Y \to \sO_E\to 0$.}
    % \FB{does my explanation work?}\textcolor{red}{T: Thank you for reply! I am still confuse why $f^{-1}(x)=E$. Do we assume $X$ is an affine cone?}) \FB{we added more details! do you agree?}
    Recall that also for Witt cohomology, showing the vanishing of $R^if_*W\mathcal{O}_{Y, \mathbb{Q}}$ is equivalent to check $R^if_*W\mathcal{O}_{E, \mathbb{Q}}$ vanishes (see \cite[Theorem 2.4]{Berthelot_Bloch_Esnault_On_Witt_vector_cohomology_for_singular_varieties}).

    \item As in the classical case, one can define the following notion. Let $k$ be an $F$-finite field, and let $X$ be a variety over $k$. We say that $X$ has \emph{$\bF_p$--pseudo--rational singularities} if for every normal variety $Y$ with a proper birational morphism $\pi \colon Y \to X$, the natural map $\pi_*\omega_Y \to \omega_X$ is an isomorphism of Cartier crystals. Note that, in contrast to the definition of pseudo-rational singularities, we do not require that $X$ has CM singularities, even up to Frobenius nilpotence.

    As usual, if $X$ has $\bF_p$-rational singularities and resolutions of singularities exist, then it also has $\bF_p$--pseudo--rational singularities as follows:
    Firstly, since we assumed the existence of resolution of singularities, every resolution of $X$ is $\bF_p$-rational.
    Here, we used the fact that every smooth variety has rational singularities \cite{CR11}.
    Next, replacing $Y$ with its resolution, we may assume that $Y$ is smooth because our aim is the surjectivity up to nilpotence of the natural injection $\pi_{*}\omega_Y\hookrightarrow \omega_X$. 
    Since $\pi$ is $\bF_p$-rational, we have $R\pi_*\omega_Y[d]\cong R\pi_*\omega_Y^{\bullet} \sim_C \omega_X^{\bullet}$
    by \autoref{thm: duality}, where $d\coloneqq \dim(Y)=\dim(X)$.
    Taking cohomology sheaves in degree $-d$ concludes the proof.
    % Indeed, by assumption it is enough to check the condition for a resolution $\pi \colon Y \to X$, where it follows from the fact that since $\mathrm{R}\pi_*\cO_Y \sim_F \cO_X$, we know by \autoref{thm: duality} that $R\pi_*\omega_Y^{\bullet} \sim_C \omega_X^{\bullet}$. Taking cohomology sheaves in degree $-\dim(X)$ concludes the proof.
    
    \iffalse
    Let $f \colon Y \to X$ be a proper morphism of Noetherian $\mathbb{F}_p$-schemes. %The sheaf $R^if_*\bZ_p$ is zero if and only if for all $n$, there exists $m \geq n$ such that the natural morphism $R^if_*\bZ/p^m\bZ \to R^if_*\bZ/p^n\bZ$ is zero.
    The sheaf $R^if_*\bQ_p$ is defined as the image of $R^if_*\bZ_p$ via the localization of the category of $\Z_p$-sheaves by the Serre subcategory of sheaves annihilated by a fixed power of $p$. In particular,  by definition, if $R^if_*\bQ_p=0$, then $R^if_*\bZ_p$ is $p^N$-torsion for some $N>0$. 
    \fi
    %\item We define the étale sheaf $\bZ_p$ on $X$ as the pro-object (see \stacksproject{0G2W}) $\{\bZ/p^n\bZ\}_{n \geq 0}$ with the usual transition maps $\bZ/p^{n + 1}\bZ \to \bZ/p^n\bZ$. If $f \colon X \to Y$ is a morphism, then $R^if_*\bZ_p$ is by definition the pro-object $\{R^if_*\bZ/p^n\bZ\}_{n \geq 0}$ (see \cite{PZ21}*{} and the references therein).
\end{enumerate}
\end{remark}

We define a Frobenius--stable version of the Grauert--Riemenschneider vanishing theorem for proper birational morphisms. 

\begin{definition} \label{def: GR_nilpotent}
    Let $X$ be a Noetherian and $F$-finite integral scheme with a unit dualising complex $\omega_X^{\bullet}$. We say that $X$ satisfies the \emph{Frobenius--stable GR vanishing} if there exists a resolution $f \colon Y \to X$ such that $R^if_*\omega_Y \sim_C 0$ for all $i>0$, where we set $\omega_Y^{\bullet} \coloneqq f^!\omega_X^{\bullet}$.
\end{definition} 

Throughout, we will implicitly use the following independence result.
\begin{lemma} \label{doest_depend_on_unit_dc}
    The above definition does not depend on the choice of a unit dualising complex on $X$.
\end{lemma}
\begin{proof}
    Let $\omega_{X, 1}^{\bullet}$ and $\omega_{X, 2}^{\bullet}$ denote two unit dualising complexes on $X$. By \stacksproj{0A7F} (whose proof extends to the non--necessarily affine case), we know that as complexes of coherent sheaves, we can write $\omega_{X, 2}^{\bullet} \cong \omega_{X, 1}^{\bullet} \otimes \mathcal{L}[m]$ for some line bundle $\mathcal{L}$ on $X$ and some $m \in \bZ$ (as mere complexes of coherent sheaves). Let $\bD_1$ be the duality functor associated to $\omega_{X, 1}^{\bullet}$ (see \autoref{thm: duality}). Since $\omega_{X, 2}^{\bullet}$ is unit, it follows from the constructions that $\bD_1(\omega_{X, 2}^{\bullet})$ is a unit Frobenius module (see \cite[Definition 2.2.1]{Bau23}). As coherent sheaves, $\bD(\omega_{X, 2}^{\bullet}) \cong \mathcal{L}^{\vee}[-m]$, so $\mathcal{L}^{\vee}$ admits the structure of a unit Frobenius module. Note that its dual does too in consequence as in \cite[Lemma 4.1.7]{Bau23} (the unit part is crucial here), and a similar computation as that of \emph{loc. cit.} shows that $\bD_1(\mathcal{L}^{\vee}[-m]) \cong \omega_{X, 1}^{\bullet} \otimes \mathcal{L}[m]$ as Cartier modules (the tensor product of a Cartier module and a unit Frobenius module is defined as in \cite[Remark 3.1.2.(d)]{Bau23}). Since $\bD_1(\mathcal{L}^{\vee}[-m]) \cong \bD_1(\bD_1(\omega_{X, 2}^{\bullet})) \cong \omega_{X, 2}^{\bullet}$, we deduce that $\omega_{X, 2}^{\bullet} \cong \omega_{X, 1}^{\bullet} \otimes \mathcal{L}[m]$ in $D^b(\Coh_C^X)$. 
    
    Set $\omega_{Y, i}^{\bullet} \coloneqq f^!\omega_{X, i}^{\bullet}$ for $i = 1, 2$. We then obtain by construction of the upper shriek functor that $\omega_{Y, 2}^{\bullet} \cong \omega_{Y, 1}^{\bullet} \otimes f^*\mathcal{L}[m]$, so $\omega_{Y, 2} \cong \omega_{Y, 1} \otimes f^*\mathcal{L}$. The result then follows from the projection formula (this formula holds in this context by duality and the fact that it holds for Frobenius modules, see \cite[Theorem 6.4.9]{Bockle_Pink_Cohomological_Theory_of_crystals_over_function_fields}).
\end{proof}
\begin{remark}\label{doesnt_depend}
    Assuming the existence of resolution of singularities in positive characteristic, it would be sufficient to check $\mathbb{F}_p$-rationality and the Frobenius-stable version of GR vanishing on $X$ starting from one resolution by \cite{CR15} (see also \cite{CR11} for varieties over a perfect field).
\end{remark}

We develop a theory of $\mathbb{F}_p$-rationality for singularities satisfying the Frobenius--stable GR vanishing theorem, analogous to the one in characteristic 0.

\begin{proposition} \label{lem: equiv_conj_GR_CM}
Let $(R,\m)$ be a $d$-dimensional Noetherian $F$-finite normal domain such that $X\coloneqq \Spec\,R$ has $\mathbb{F}_p$-rational singularities.
Let $f \colon Y \to X$ be a proper birational morphism, where $Y$ is regular.
Then for all $i > 0$, we have that $R^{d-i}f_*\omega_Y \sim_{C} 0$ if and only if $H^{i}_{\m}(\sO_X)\sim_{F} 0$. 

In particular, $X$ satisfies the Frobenius stable GR vanishing theorem if and only if $X$ is CM up to nilpotence.
\end{proposition}

\begin{proof}
By local duality \cite{Bau23}*{Proposition 4.4.5}, we have
\[
H^i_{\m}(\sO_X)\sim_{F}0\iff R^{-i}\Hom(\sO_X,\omega^{\bullet}_X)\sim_{C}0.
\]
Since $Rf_{*}\sO_Y\sim_{F}\sO_X$ by the $\mathbb{F}_p$-rationality assumption, we obtain that \[ R^{-i}\Hom(\sO_X,\omega^{\bullet}_X)\sim_{C} R^{-i}\Hom(Rf_{*}\sO_Y,\omega^{\bullet}_X)\sim_{C} 
R^{-i}f_{*}\Hom(\sO_Y,\omega_Y[d])=R^{d-i}f_{*}\omega_Y.\]
Thus, we conclude the proof of the statement before ``In particular''. The last assertion follows from \autoref{lem: equivalent_def_CM_up_to_nilp}.
\end{proof}

% \begin{proposition} \label{lem: equiv_conj_GR_CM2}
% Let $X$ be an integral Noetherian $F$-finite normal scheme with $\mathbb{F}_p$-rational singularities.
% Let $f \colon Y \to X$ be a proper birational morphism, where $Y$ is regular.
% Then $R^if_*\omega_Y \sim_{C} 0$ for $i>0$ if and only if $X$ is CM up to nilpotence.
% \end{proposition}

% \begin{proof}
% Let $d = \dim(X) = \dim(Y)$. Since $\mathbf{R}f_{*}\sO_Y\sim_{F}\sO_X$, we obtain \[ \omega_X^{\bullet}\sim_{C}\mathbf{R}\HHom(\mathbf{R}f_{*}\sO_Y, \omega_X^{\bullet})\sim_{C} \mathbf{R}f_{*}\mathbf{R}\HHom(\sO_Y, \omega_Y^{\bullet})\sim_{C}\mathbf{R}f_*\omega_Y^{\bullet} \sim_{C} \mathbf{R}f_*\omega_Y[d] \]
% by \autoref{thm: duality}.
% Therefore, Frobenius--stable GR vanishing coincides with $\omega_X^{\bullet}$ being supported in a single degree as a Cartier crystals. Thus, we conclude by \autoref{lem: equivalent_def_CM_up_to_nilp}.
% % Let $x \in X$ and denote $W:=f^{-1}(x)$.
% % By hypothesis and Leray spectral sequence for local cohomology, we have that $H^i_x(\mathcal{O}_X) \cong H^i_V(\mathcal{O}_Y) \cong (R^{n-i}f_*\omega_Y)_x$ which vanishes.
% \end{proof}
\iffalse
\begin{remark}
    As the proof shows, we only used that $Y$ was CM up to nilpotence.
\end{remark}
\fi

A typical counter-example to the usual GR vanishing theorem is given by taking cones over smooth varieties for which the Kodaira vanishing theorem fails, see for example \cite[Example 3.11]{HK15}. It turns out that these do not give rise to counter-examples to Frobenius--stable GR vanishing. Although this has already been proven in \cite{BBLSZ23}*{Proposition 5.18} (and discussion afterwards), we give a direct proof without mentioning perverse sheaves.

\begin{example} \label{ex: Frobenius-stable-GR-cone}
Let $k$ be an $F$-finite field. We follow the notations from \cite{kk-singbook}*{Section 3.1}.
Given a normal projective variety $Z$ over $k$ and an ample line bundle $L$, we define $X=C_a(Z, L)  \coloneqq \Spec_k \bigoplus_{n \geq 0} H^0(Z, L^{\otimes n})$ to be the affine cone over $Z$ (with respect to $L$) and we denote by $v$ its vertex.

\begin{claim} \label{prop: GR-cones}
If $Z$ is regular, then $f \colon Y\coloneqq \Spec_Z \bigoplus_{m \geq 0} L^{\otimes m} \to X$ is a log resolution
and $R^if_*\omega_Y \sim_{C} 0$ for $i>0$.
\end{claim}

\begin{proof}[Proof of the Claim]
Fix an integer $e > 0$, and let $E$ be the exceptional divisor of $f$, which is isomorphic to $Z$.
%Let $\mathcal{I}$ be the invertible ideal sheaf associated to $E$. 
We have $\MO_{Y}/\MO_Y(-mE) \cong \mathcal{O}_{mE}$ for every $m \geq 1$.
By the theorem on formal functions (\cite{stacks-project}*{Tag \href{https://stacks.math.columbia.edu/tag/02OC}{02OC}}) we have the following commutative diagram, where the vertical arrows are isomorphisms:

\begin{center}
	\begin{tikzcd}
	 (F^e_*R^if_*\omega_Y)^{\wedge}_v \ar[r, "\Tr^e"] \ar[d, "\cong"] & (R^if_*\omega_Y)^{\wedge}_v \ar[d, "\cong"] \\
	  \varprojlim_m H^i(mE, F^e_*\omega_Y \otimes \mathcal{O}_{mE}) \ar[r] & \varprojlim_m H^i(mE, \omega_Y \otimes \mathcal{O}_{mE})
	\end{tikzcd}
\end{center}
The restriction of $\pi \colon Y \to Z$ to $mE \to Z$ gives a finite morphism, so we can compute 
$H^i(mE, \omega_Y \otimes \mathcal{O}_{mE}) \cong H^i(Z, \pi_*(\omega_Y \otimes \mathcal{O}_{mE}))$. Fix an open cover $\{U_i\}_i$ of $Z$ over which $L|_{U_i}$ admits a nowhere vanishing section $t_i$. We define $\eta \in H^0(Y,\Omega^1_{Y/Z}(\log E))$ as being given by $d\log(\pi^*t_i)$ over $U_i$ (one readily verifies that $\eta$ is well--defined and does not depend on the choice of $\{t_i\}_i$). By \cite{kk-singbook}*{Proposition 3.14.3}, the morphism $\omega_Z \to \pi_*\omega_Y(E)$ given by $\omega \mapsto \pi^*\omega \wedge \eta$ induces an isomorphism $\pi^*\omega_Z \to \omega_Y(E)$. By the projection formula, we have that
$$\pi_*(\omega_Y(-mE)) \cong \omega_Z \otimes \pi_*\cO_Y(-(m+1)E) \cong \bigoplus_{n \geq m+1} \omega_Z \otimes L^{\otimes n}.$$ 
As $\pi$ is an affine morphism, the higher direct images $R^i\pi_*$ vanish for $i>0$ and therefore, we have:
\begin{equation*} \label{eq1}
\begin{split}
\pi_*(\omega_Y \otimes \mathcal{O}_{mE}) & \cong \coker \big({ \pi_*(\omega_Y \otimes \cO_Y(-mE)) \hookrightarrow \pi_*\omega_Y )\big)}  \\
& \cong \coker {\big( \bigoplus_{n \geq m+1} (\omega_Z \otimes L^{\otimes n}) \hookrightarrow \bigoplus_{n \geq 1} \: (\omega_Z \otimes L^{\otimes n}) \big) }  \cong \bigoplus_{1 \leq n \leq m}(\omega_Z \otimes L^{\otimes n}), 
\end{split}
\end{equation*}
Similarly, we have

\begin{equation*} \label{eq2}
\begin{split}
\pi_*(F^e_*\omega_Y \otimes \mathcal{O}_{mE}) & \cong \coker \big({ \pi_*(F^e_*\omega_Y \otimes \cO_Y(-mE)) \to \pi_*F^e_*\omega_Y )\big)}  \\
& \cong \coker {\big(\pi_*F^e_*(\omega_Y \otimes \cO_Y(-p^emE)) \to \pi_*F^e_*\omega_Y } \big) \\
%& \cong \coker {\big(F^e_*(\omega_Z \otimes \pi_*\cO_Y(-(p^em+1)E))) \to F^e_*(\omega_Z \otimes \pi_*\cO_Y(-E)) } \big) \\
& \cong \coker {\big( \bigoplus_{n \geq p^em+1} F^e_*(\omega_Z \otimes L^{\otimes n}) \to \bigoplus_{n \geq 1} F^e_*(\omega_Z \otimes L^{\otimes n}) \big) } \\
& \cong \bigoplus_{1 \leq n \leq mp^e}F^e_*(\omega_Z \otimes L^{\otimes n}). 
\end{split}
\end{equation*}
We claim that the diagram 
\[ 
    \begin{tikzcd}
	   \pi_*(F_*^e\omega_Y \otimes \mathcal{O}_{mE})) \ar[r, "\Tr^e"] \ar[d, "\cong"] & \pi_*( \omega_Y \otimes \cO_{mE})) \ar[d, "\cong"] \\
	   \bigoplus_{1 \leq n \leq p^em} F^e_*(\omega_Z \otimes L^{\otimes n}) \ar[r,"\oplus \psi_n"]  & \bigoplus_{1 \leq n \leq m} \omega_Z \otimes L^{\otimes n} 	,
    \end{tikzcd} \] commutes, where the bottom arrow is the direct sums indexed by $n$ of the following morphisms:
\begin{equation*}
\begin{cases} 
\mbox{if } n=p^ek, & \psi_n=\Tr^e_Z \otimes L^{\otimes k} \colon F^e_*(\omega_Z \otimes L^{\otimes p^ek}) \to \omega_Z \otimes L^{\otimes k};   \\ \mbox{if }n \mbox{ is not divisible by } p^e , & \psi_n=0. 
\end{cases}    
\end{equation*}

This can be shown locally on $Y$, so take an open $U \inc Z$ and a nowhere vanishing section $t$ of $L$ on $U$. Then the isomorphism $\pi_*\omega_Y(E) \cong \bigoplus_{m \geq 0}(\omega_Z \otimes L^{\otimes m})$ sends by construction an element $\sum_{m \geq 0} \left(\pi^*(\omega_m) \wedge \frac{dt}{t}\right) t^m \in \pi_*\omega_Y(E)$ to $\sum_{m \geq 0} \omega_m \otimes t^m$. Recall that the Cartier operator $C$ commutes with pullbacks, i.e. $C(\pi^*(\psi))=\pi^*(C(\psi))$ for all closed differential forms $\psi$, wedge products, fixes logarithmic forms and satisfies $C^e(t^mdt) = t^{\frac{m + 1}{p^e}}dt$ on $\bA^1$ for all $m \geq 0$ (we set the notation that $t^r = 0$ if $r \notin \bZ$). Since $\Tr = C$ on top--forms, we have that \[ \Tr^e\left(\sum_{m \geq 0} \left(\pi^*(\omega_m) \wedge \frac{dt}{t}\right) t^m\right) = \sum_{m \geq 0} \left (\pi^*\Tr^e(\omega_m) \wedge \frac{dt}{t}\right)t^{m/p^e} \] is sent to $\sum_{m \geq 0} \Tr^e(\omega_m) \otimes t^{m/p^e}$ via the isomorphism $\pi_*\omega_Y(E) \cong \bigoplus_{m \geq 0}(\omega_Z \otimes L^{\otimes m})$. The rest of the computation is now straight--forward. \\

It follows from the explicit formula above of the map $\oplus_n \psi_n$ and Serre vanishing that the iterated trace $\Tr^e \colon H^i(Z, \pi_*(F^e_*\omega_Y \otimes \cO_{mE})) \to H^i(Z, \pi_*(\omega_Y \otimes \cO_{mE}))$ vanishes for $e \gg 0$ independent of $m$, so the proof is complete. \qedhere

\iffalse
We then obtain a commutative diagram
\begin{center}
\begin{tikzcd}
	H^i(Z, \pi_*(F_*^e\omega_Y \otimes \mathcal{O}_{mE})) \ar[r, "\Tr^e"] \ar[d, "\cong"] &  H^i(Z,\pi_*( \omega_Y \otimes \cO_{mE})) \ar[d, "\cong"] \\
	\bigoplus_{1 \leq n \leq p^em} H^i(Z, F^e_*(\omega_Z \otimes L^{\otimes n})) \ar[r,"\oplus \psi_n"]  &  \bigoplus_{1 \leq n \leq m} H^i(Z, \omega_Z \otimes L^{\otimes n}),
 \end{tikzcd}
\end{center}
where the bottom arrow is the direct sums indexed by $n$ of the following morphisms:
\begin{equation*}
\begin{cases} 
\mbox{if } n=p^ek \mbox { and } p \mbox { does not divide } k, & \psi_n=\Tr^e_Z \otimes L^{\otimes k} \colon F^e_*(\omega_Z \otimes L^{\otimes p^ek}) \to \omega_Z \otimes L^{\otimes k};   \\ \mbox{if }n \mbox{ is not divisible by } p^e , & \psi_n=0. 
\end{cases}    
\end{equation*}

\fi

\end{proof}

\end{example}

\iffalse
We have an analogue of \cite{Kov00}*{Theorem 1.1}.

\begin{proposition}
        Let $f \colon Y \to X$ be a proper morphism of $\mathbb{F}_p$-schemes.
        Suppose $Y$ has rational singularities up to nilpotence.
        Assume \autoref{conj: GR_nilpotent}.
        If $\mathcal{O}_X \to \mathbf{R}f_*\mathcal{O}_Y$ splits in $D(\QCoh_X)$, then $Y$ has rational singularities up to nilpotence.
\end{proposition}

\begin{proof}
\FB{to write in detail}
The proof of Kovacs uses Grothendieck duality, adjunction and GR vanishing, which works in our setting.    
\end{proof}
\fi

\iffalse
We need the following CM criterion in the category of crystals.

\begin{proposition} \label{prop: CM-criterion}
    Let $ f \colon Y \to X$ be a proper birational morphism.
    Assume that
    \begin{enumerate}
        \item $Y$ is normal and CM up to nilpotence;
        \item \autoref{conj: GR_nilpotent} holds;
        \item $\pi^{KF}$ effective; 
    \end{enumerate}
    Then $X$ is CM up to nilpotence.
\end{proposition}

\begin{proof}
    We use the method of the 2 exact sequences. 
    We have the following commutative diagram:
    
\end{proof}
\fi

\subsection{$\mathbb{F}_p$-rationality versus $W\mathcal{O}$-rationality} \label{ss-comparison}

In this subsection, we compare the notion of $\mathbb{F}_p$-rationality with $\mathbb{Q}_p$-rationality and $W\cO$-rationality. We start with introducing the (temporary) notion of $\mathbb{Z}/p^n \mathbb{Z}$-rational singularities.

\begin{definition} \label{rem: b}
We say that a projective morphism $f \colon Y \to X$ is $(\bZ/p^n\bZ)$-rational if the natural morphism $(\bZ/p^n\bZ)_X \to Rf_*(\bZ/p^n\bZ)_Y $ is an isomorphism. 
Similarly, an integral $\bF_p$-scheme $X$ has $(\bZ/p^n\bZ)$-rational singularities if it admits a $(\bZ/p^n\bZ)$-rational resolution.
\end{definition}

\begin{remark} The fact that $f$ is $\mathbb{Z}/p^n \mathbb{Z}$-rational is equivalent to to the nilpotence of the natural Frobenius action on $f_*W_n\cO_Y/W_n\cO_X$ and on $R^if_*W_n\cO_Y$ for all $i>0$ (this follows from the case $n = 1$, and the same induction argument as in the proof of \cite[Proposition 9.5.6]{Bhatt_Lurie_RH_corr_pos_char}).
\end{remark}

    % \item\label{rem: enough to check vanishing of the exceptional (c)} Recall (\cite{PZ21}*{Section 3.3} and the references therein) the notions of $\bZ_p$ and $\bQ_p$-sheaves. Note that if $f \colon Y \to X$ is a proper morphism, the sheaf $R^if_*\bZ_p$ (resp.~$R^if_*\bQ_p$) is zero if and only if for all $n$, there exists $m \geq n$ such that the natural morphism $R^if_*\bZ/p^m\bZ \to R^if_*\bZ/p^n\bZ$ is zero (resp.~the image if $p^l$-torsion for a fixed $l \geq 0$ independent of $n$ and $m$).
    %\item\label{rem: enough to check vanishing of the exceptional (c)} 
    
    We refer to \cite{PZ21}*{Section 3.3} and references therein for the definitions of $\bZ_p$, $\bQ_p$-sheaves and the definition of their higher direct images. One can then define $\bQ_p$-rational morphisms and singularities as in \autoref{rem: b}. Note that in this case, we only need the existence of a $\bQ_p$-rational \emph{quasi-resolution} (this follows from the case of $W\cO$-rational singularities (see \cite{CR12}*{Corollary 4.4.7}) and \cite{PZ21}*{Lemma 3.19}). 
    In particular, note that if $R^if_*\mathbb{Z}/p^n\mathbb{Z}=0$ for all $n>0$ for a resolution, then $R^if_*\mathbb{Q}_p=0$.

\begin{lemma}
    Let $X$ be a variety over an $F$-finite field $k$, and let $\pi \colon Y \to X$ be a resolution. 
    Then the following are equivalent: 

    \begin{enumerate}
        \item\label{itm:eq_F_p} $X$ has $\bF_p$-rational singularities;
        \item\label{itm:eq_Z/pn} $X$ has $(\bZ/p^n\bZ)$-rational singularities for some integer $n>0$;
        \item\label{forall} $X$ has $(\bZ/p^n\bZ)$-rational singularities for all integers $n > 0$.
        %\item\label{itm:eq_Z_p} $X$ has $\bZ_p$-rational singularities.
    \end{enumerate} 

    Furthermore, these conditions imply that $X$ has $\bQ_p$-rational singularities.
\end{lemma}

\begin{proof}
    We show that \autoref{itm:eq_F_p} implies \autoref{forall} by induction on $n$. If $n = 1$, there is nothing to show, so assume that $n > 1$. The result then follows from the long exact sequence for higher pushforwards induced from \begin{equation}\label{eq:ses_Z/pn}
        0 \to \bF_{p, Y} \to (\bZ/p^n\bZ)_Y \to (\bZ/p^{n - 1}\bZ)_Y \to 0,
    \end{equation} the induction hypothesis, and the fact that $\pi_*(\bZ/p^n\bZ)_Y \to \pi_*(\bZ/p^{n - 1}\bZ)_Y$ is surjective. Although this last fact holds in this case since $R^1\pi_*(\bF_{p,Y}) = 0$ by assumption, this surjectivity holds in fact in general for any morphism, as shown by evaluating these sheaves on an étale open.
    
    Since \autoref{forall} implies \autoref{itm:eq_Z/pn} tautologically, we are left to show that \autoref{itm:eq_Z/pn} implies \autoref{itm:eq_F_p}. Let $d \geq 1$ be the smallest integer such that $R^j\pi_*\bF_{p, Y} = 0$ for all $j \geq d$. We want to show that $d = 1$, so assume by contradiction that $d > 1$. Note that as in the beginning of the proof, we deduce from $R^j\pi_*\bF_{p, Y} = 0$ that $R^j\pi_*(\bZ/p^l\bZ)_Y = 0$ for all $l \geq 1$ and $j \geq d$. In particular, the morphism $\theta_n \colon R^{d - 1}\pi_*(\bZ/p^n\bZ)_Y \to R^{d - 1}\pi_*\bF_{p, Y}$ coming from the short exact sequence \[ 0 \to (\bZ/p^{n - 1}\bZ)_Y \to (\bZ/p^n\bZ)_Y \to \bF_{p, Y} \to 0 \] is surjective.
  %  On the other hand, by \autoref{rem: enough to check vanishing of the exceptional}\autoref{rem: enough to check vanishing of the exceptional (c)},
    %by definition of $\bZ_p$-rationality, 
  %  the morphism $\theta_n$ is zero for $n \gg 0$.
    By hypothesis, $R^{d-1}\pi_* (\mathbb{Z}/p^n\mathbb{Z})_Y=0$ and thus $R^{d-1}\pi_*\bF_{p, Y} = 0$, reaching the desired contradiction. 
    
    We are left to show that $\bF_{p, X} \to \pi_*\bF_{p, Y}$ is an isomorphism. 
    From the commutative diagram 
    \begin{center}
        \begin{tikzcd}
           0 \arrow[r]  & \pi_*(\mathbb{Z}/{p^{n-1}\mathbb{Z}})_Y \ar[r] & \pi_*(\mathbb{Z}/p^n\mathbb{Z})_Y \ar[r] & \pi_*\mathbb{F}_{p, Y} \arrow[r] & R^1\pi_*(\mathbb{Z}/{p^{n-1}\mathbb{Z}})_Y=0 \\
           0 \arrow[r] & (\mathbb{Z}/p^{n-1}\mathbb{Z})_X \arrow[r] \arrow[u] & (\mathbb{Z}/p^n\mathbb{Z})_X \arrow[r] \arrow[u, "\cong"] & \mathbb{F}_{p, X} \arrow[r] \arrow[u, "\psi"] &  0,
        \end{tikzcd}
    \end{center}

\noindent we deduce that $\psi$ is surjective.
As $\psi$ is clearly injective, we conclude $\bF_{p, X} \to \pi_*\bF_{p, Y}$ is an isomorphism.
    %Let $C_n \coloneqq \coker\left((\bZ/p^n\bZ)_X \to \pi_*(\bZ/p^n\bZ)_Y\right)$. Since each map $\pi_*(\bZ/p^n\bZ)_Y \to \pi_*(\bZ/p^{n - 1}\bZ)$ is surjective, so are the maps $C_n \to C_{n - 1}$. By $\bZ_p$-rationality, $\varprojlim_n C_n = 0$, so each $C_n = 0$, proving in particular $\bF_p$-rationality.

    Finally, the fact that $\mathbb{F}_p$-rationality implies $\bQ_p$-rationality is immediate by definition (see the discussion before this lemma). %\autoref{rem: enough to check vanishing of the exceptional} \autoref{rem: enough to check vanishing of the exceptional (c)}).
\end{proof}

For the theory of $W\mathcal{O}$-singularities, we refer to \cites{BE08, CR12}: here we recall the notion of $W\cO$-rationality.

\begin{definition}
   A variety $X$ over a perfect field $k$ has \emph{W$\mathcal{O}$-rational} singularities if for every (equiv.~some) quasi-resolution $f \colon Y \to X$, the following two conditions are satisfied:
   \begin{enumerate}
       \item $f$ induces an isomorphism $W \mathcal{O}_{X,\bQ} \cong f_* W \mathcal{O}_{Y,\mathbb{Q}} $;
       \item $R^if_*W \cO_{Y,\mathbb{Q}}=0$ for $i>0$. 
   \end{enumerate}
\end{definition}
\begin{lemma}[\cite{PZ21}*{Lemma 3.19}] \label{prop: rat-up-nil-implies-Wrat}
     If a variety $X$ over a perfect field $k$ has $W\cO$-rational singularities, then it has $\mathbb{Q}_p$-rational singularities.
\end{lemma} 

The following diagram summarises the implications:
\begin{equation*}
\fbox{$\bF_p$-rational} \Leftrightarrow \fbox{$(\bZ/p^n\bZ)$-rational for all $n > 0$} \Rightarrow \fbox{$\bQ_p$-rational}
\Leftarrow \fbox{$W\cO$-rational}
\end{equation*}

The following examples show that $W\mathcal{O}$-rationality and $\mathbb{F}_p$-rationality are distinct notions, even for log canonical singularities.

\begin{example}\label{example1}
    Let $E$ be a supersingular elliptic curve over a perfect field $k$ and let $X$ be the affine cone over $E$.
    Then $X$ is log canonical and it has strongly $\mathbb{F}_p$-rational singularities, as $E$ is supersingular (see \autoref{rem: enough to check vanishing of the exceptional}). 
    On the other hand, $X$ is not $W\mathcal{O}$-rational by \cite{CR12}*{Theorem 4.7.4}.
\end{example}

\begin{example}
    Let $S$ be a globally $F$-split smooth Enriques surface over a perfect field of characteristic $p=2$ (called singular in \cite{BM76} or ordinary in \cite{LT22}).
    Let $X$ be the affine cone over $S$, which is a log canonical 3-fold singularity by \cite{kk-singbook}*{Lemma 3.1}.
    Since $S$ is globally $F$-split, the Frobenius map acts on $H^1(S, \mathcal{O}_S) = k$ bijectively, and thus $X$ is not $\mathbb{F}_p$-rational.
    Nevertheless, we have $H^2(S, W\mathcal{O}_{S, \bQ}) = 0$ by \cite{Ill79}*{Proposition 7.3.2} and $H^1(S, W\mathcal{O}_S) = 0$ by \cite{Ill79}*{Proposition 7.3.6}, and thus $X$ has $W\cO$-rational singularities by \cite{CR12}*{Theorem 4.7.4}.
\end{example}

\section{$\mathbb{F}_p$-rationality of Fano morphisms and of klt singularities} \label{s: global-local}

In this section, we discuss a general strategy to prove that klt singularities in positive characteristic of dimension $n$ are $\mathbb{F}_p$-rational, assuming the MMP and the vanishing of the Frobenius on the cohomology of the higher direct images of Fano type morphism. 
As a corollary, we prove that $F$-finite 3-folds klt-type singularities are $\mathbb{F}_p$-rational.

\subsection{Reduction from global-to-local}

We implement a global-to-local principle: we prove that the vanishing up to Frobenius nilpotence of higher direct images of the structure sheaf of Fano type morphisms of dimension at most $n-1$ together with MMP imply $\mathbb{F}_p$-rationality of $n$-dimensional klt type singularities.
We propose the following global and local conjectures.

\begin{conjecture} \label{conjecture}
 Let $k$ be an $F$-finite field of characteristic $p>0$.   
    \begin{enumerate}
        \item \label{conj: vanishing_nilpotent} Let $f \colon X \to Y$ be a Fano type morphism over $k$.
        Then the following vanishing of Frobenius modules holds: $R^if_*\cO_X \sim_{F} 0$ for $i>0$.
        \item \label{conj: F_p-rationality} Let $(R, \mathfrak{m})$ be a local $k$-algebra essentially of finite type. If $X=\Spec(R)$ is of klt type, then it has $\mathbb{F}_p$-rational singularities.
    \end{enumerate}
\end{conjecture}

\begin{theorem} \label{thm: F_p-rationality-Fano}
    Let $k$ be an $F$-finite field of characteristic $p>0$. 
    Let $(R, \mathfrak{m})$ be a local $k$-algebra essentially of finite type of dimension $n$ such that $X=\Spec(R)$ is of klt type.
    Suppose that all the following statements hold:
    \begin{itemize}
        \item existence of a log resolution for every birational model of $X$;
        \item existence of the birational MMP as in \autoref{thm: kollar-notqfactorial-4folds} in dimension $n$;
        \item \autoref{conjecture} \autoref{conj: vanishing_nilpotent} and \autoref{conj: F_p-rationality} are valid in dimension $\leq n-1$.
    \end{itemize}
    Then $X$ has $\mathbb{F}_p$-rational singularities.
\end{theorem}
\begin{remark} \label{label:remark_thm_induction_rationality}
If furthermore we assume the MMP (not only the birational one) in dimension $n-1$, then \cite{GNT19}*{Proposition 2.15} shows that it is sufficient to prove the vanishing up to nilpotence for varieties of Fano type in dimension at most $n-1$ (and not for any Fano type morphism).
For a precise argument, see the proof of \autoref{lem: Fano-type}.
\end{remark} 

The idea of the proof of \autoref{thm: F_p-rationality-Fano} is natural: we start with a log resolution, we run a relative MMP and we control that $\mathbb{F}_p$-rationality is preserved by each of the steps of the MMP.
We begin with the case of (pl-)birational contractions.

\begin{lemma} \label{lem: rationality_contraction}
    Let $(X,S+\Delta)$ be a dlt pair over an $F$-finite field $k$, where $S$ is a prime divisor.
    Let $g \colon X \to Z$ be a  birational contraction such that
    \begin{enumerate}
        \item $-(K_X+S+\Delta)$ is $g$-ample and 
        \item $-S$ is a $g$-ample $\mathbb{Q}$-Cartier divisor.
    \end{enumerate}
    Suppose that \autoref{conjecture}.\autoref{conj: vanishing_nilpotent} holds in dimension $\leq n-1$.
    Then $R^ig_*\cO_X \sim_F 0$ for $i>0$.
    In particular, if $X$ has $\mathbb{F}_p$-rational singularities, then so does $Z$.
\end{lemma}

\begin{proof}
    For any $m \geq 0$, we see $mS$ as both a Weil divisor and a closed subscheme on $X$. In particular, $(mS)_{\red} = S$.
    
    We take a sufficiently large $m \gg 0$, and consider the following short exact sequence of Frobenius crystals:
    \begin{equation}
        0\to \mathcal{O}_X(-mS) \to \mathcal{O}_X \to \mathcal{O}_{mS} \to 0. \label{exact 4.4}
    \end{equation} 
    By adjunction, $g^\nu \colon S^{\nu} \to g(S)$ is a Fano type contraction. Therefore we have that $R^ig^{\nu}_*\mathcal{O}_{S^{\nu}} \sim_F 0$ for $i>0$ by hypothesis.
    As $S^{\nu} \to S$ is a universal homeomorphism by \cite{HW20}*{Lemma 2.1} 
    and $S \to mS$ is a universal homeomorphism by \stacksproj{0CNF}, we deduce from \autoref{universal_homeomorphism_induces_isomorphism_of_crystals} that
    \[ R^ig_*\mathcal{O}_{mS} \sim_F 0 \text{ for all } i>0.\] 
    
    On the other hand, since $-S$ is $g$-ample and $m \gg 0$ , we have $R^if_*\cO_X(-mS) = 0$ for all $i > 0$ by Serre vanishing. 
    We then conclude from the long exact sequence of \eqref{exact 4.4} that $R^if_{*}\sO_X\sim_F 0$ for all $i>0$.
\end{proof}

\begin{lemma}\label{lem: rationality_flips}
    Let $(X, S+T+\Delta)$ be a dlt pair over an $F$-finite field $k$, where $S$ and $T$ are prime divisors.
    Let 
    \[ \begin{tikzcd}
        & W \arrow[ld, "\psi"'] \arrow[rd, "\psi^+"] &                \\
        X \arrow[rd, "f"'] \arrow[rr, "\phi", dashed] &                                            & X^+ \arrow[ld, "f^+"] \\
        & Z &                     
    \end{tikzcd} \] be a commutative diagram of birational maps where
    \begin{enumerate}
        \item $\psi$ is a log resolution of $(X, S+T+\Delta)$ and the morphisms $f, f^+$ are small;
        \item $-(K_X+S+T+\Delta)$ is $f$-ample and $\phi$ is the $(K_X+S+T+\Delta)$-flip;
        \item every irreducible component of $\lfloor S+T+\Delta \rfloor$ is $\mathbb{Q}$-Cartier;
        \item \label{item: anti-ample} $-S$ is $f$-ample and $-T^{+}$ is $f^{+}$-ample, where $T^{+}$ is the strict transform of $T$ on $X^+$.
        \item \label{item: codim_supp_0} $\dim_{\crys}(R^i\psi^{+}_*\cO_W) = 0$ for all $i > 0$;
    \end{enumerate}
    Assume that \autoref{conjecture}.\autoref{conj: vanishing_nilpotent} holds in dimension $\leq n-1$.
    If $X$ has $\mathbb{F}_p$-rational singularities, then so does $X^{+}$.
\end{lemma}

\begin{proof}
    We claim that it is sufficient to show that $R^if^+_*\cO_{X^+} \sim_F 0$ for all $i > 0$.
    Let us show this. By \autoref{lem: rationality_contraction}, we have that $R^if_*\cO_X \sim_F 0$ for $i > 0$ and, as $X$ has $\mathbb{F}_p$-rational singularities by hypothesis, we also have $R^i\psi_*\cO_W \sim_F 0$ for $i>0$. By the Leray spectral sequence, we deduce $R^i(f \circ \psi)_*\cO_W \sim_F 0$. Since $\dim_{\crys}(R^i\psi^+_*\cO_W) = 0$, we know by \cite{Bau23}*{Lemma 2.2.17} that there exists a coherent Frobenius module $\cN_i$ on $X^+$, supported at finitely many closed points, such that $\cN_i \sim_F R^i\psi^+_*\cO_W$. In particular, $R^jf^+_*R^i\psi^+_*\cO_W \sim_F R^jf^+_*\cN_i \sim_F 0$ for all $i, j > 0$. We obtain from the Leray spectral sequence that \[ 0 \sim_F R^i(f^+ \circ \psi^+)_*\cO_W \sim_F f^+_*R^i\psi^+_*\cO_W, \] so $R^i\psi^+_*\cO_W \sim_F 0$. Hence, $X^+$ has $\mathbb{F}_p$-rational singularities, and the claim is proven.

    If we consider $D \coloneqq S+ \Delta -\varepsilon\lfloor(S+\Delta)\rfloor$ for $\varepsilon>0$ sufficiently small, the pair $(X, T + D)$ is plt and $-(K_X+T+D)$ is ample over $Z$. 
    As $\psi$ is a log resolution, we can apply the same proof as in \cite{GNT19}*{Lemma 2.8} to find an ample effective $\bQ$-divisor $H \geq 0$ on $X$ 
    \begin{itemize}
        \item $K_X + T + D + H \sim_{\bQ, Z} 0$;
        \item $(X, T+D+H)$ is plt.
    \end{itemize}   
    Then $K_{X^+} + T^+ + D^+ + H^+ \sim_{\bQ, Z} 0$ and, as $f$ and $f^+$ are small, the pair $(X^+, T^+ + D^+ + H^+)$ is still plt.

    Let $Y \coloneqq (T^+)^{\nu} \to T^+$ be the normalisation. 
    Since $T^+$ is regular in codimension 1 by \cite{kk-singbook}*{Theorem 2.31}, we have that $Y \to T^+$ is an isomorphism in codimension $1$. By adjunction, the pair $(Y, D_Y^+ + H_Y^+)$ is then klt and $K_Y + D_Y^+ + H_Y^+ \sim_{\bQ, Z} 0$. 
    Since $Y \to f^+(Y)$ is birational, $H_Y^+$ is automatically $f^+$-big. Hence we can write $H_Y^+ \sim_{\bQ, Z} A + E$, where $A$ is $f^+$-ample and $E$ is an effective $\mathbb{Q}$-divisor. 
    For $0 < \epsilon \ll 1$, $(Y, D_Y^+ + (1 - \epsilon)H^+ + \epsilon E)$ is still a Fano type contraction over $Z$. 
    Hence, $R^if_*^+\cO_Y \sim_F 0$ for all $i > 0$ by hypothesis.
    
    Now, let $m \gg 0$ be such that $-mT^+$ is Cartier and $R^ig^+_*\cO_{X^+}(-mT^+) = 0$ for all $i > 0$ (recall that $-T^+$ is $f$-ample). We have a short exact sequence of Frobenius modules: 
    \begin{equation} \label{eq: temp}
        0 \to \cO_{X^+}(-mT^+) \to \cO_{X^+} \to \cO_{mT^+} \to 0.    
    \end{equation}
    
    As in the proof of \autoref{lem: rationality_contraction}, the morphism $Y \to mT^+$ is a universal homeomorphism, and hence $R^if^+_*\cO_{mT^+} \sim_F 0$ for all $i > 0$. Thus, by taking the long exact sequence associated to \autoref{eq: temp}, we deduce that $R^if^+_*\cO_{X^+} \sim_F 0$ for all $i > 0$ and the proof is complete.
\end{proof}

\begin{proof}[Proof of \autoref{thm: F_p-rationality-Fano}]
    We work by induction on the dimension and we suppose that klt type singularities of dimension at most $(n-1)$ are $\mathbb{F}_p$-rational. 
    Let $\Delta \geq 0$ be a boundary such that $(X, \Delta)$ is klt and let $\pi \colon Y \to (X, \Delta)$ be a log resolution whose exceptional divisor supports a $\pi$-ample divisor, whose existence is guaranteed by \cite{KW21}*{Theorem 1}. 
    By assumption, we can run a $(K_Y+\pi^{-1}_*\Delta+\Ex(\pi))$-MMP as in \autoref{thm: kollar-notqfactorial-4folds} which terminates with $X$ (as explained in \autoref{cor: 4-fold-klt-base}).
    We have to control at each step of the MMP that $\mathbb{F}_p$-rationality is preserved.
    In the case of a divisorial contraction, we apply \autoref{lem: rationality_contraction}. 
    In the case of a flip we apply \autoref{lem: rationality_flips}, where
    \begin{enumerate}
        \item the assumption \autoref{item: codim_supp_0} is satisfied by induction on the dimension and \autoref{conjecture}.\autoref{conj: F_p-rationality},
        \item the assumption \autoref{item: anti-ample} is satisfied: We consider $E_{j_1}^j$ and $E_{j_2}^j$ in \autoref{thm: kollar-notqfactorial-4folds} as $T$ and $S$ in \autoref{lem: rationality_flips}, respectively.
        Then the strict transform $-E_{j_1}^{j+1}$ of $-E_{j_1}^{j}$ 
        is ample over the base of the flip because a $(K_Y+\pi_*^{-1}\Delta+\Ex(\pi))$-flip is a  $(-E_{j_1}^j)$-flip by \autoref{thm: kollar-notqfactorial-4folds}.\autoref{non_Q_fact a}. \qedhere
     \end{enumerate}
\end{proof}

%\begin{lemma} \label{lem: geom-klt}
%    Let $(X, \Delta,x )\coloneqq(\Spec(R), \Delta, \mathfrak{m})$ be a klt pair over a perfect field $k$, admitting a log resolution $f \colon Y \to (X, \Delta)$. 
 %   The localisation $(X_\mathfrak{p}, \Delta_{\mathfrak{p}})$ of $(X, \Delta)$ at a prime $\mathfrak{p} \in \Spec(R)$ is geometrically klt.
%\end{lemma}

%\begin{proof}
 %  Write $K_Y+f_*^{-1}\Delta=f^*(K_X+\Delta)+\sum_i a(E_i, X, \Delta) E_i$, where by the klt hypothesis we have $a(E_i, X, \Delta)>-1$.  As $k$ is perfect, the base change $Y_{\overline{\mathfrak{p}}} \to (X_{\overline{\mathfrak{p}}}, \Delta_{\overline{\mathfrak{p}}})$ remains a log resolution.    As $a(E_i, X, \Delta)>-1$, we deduce that $a(E_{i, \overline{\mathfrak{p}}}, X_{\overline{p}}, \Delta_{\overline{p}})>-1$, we conclude that $(X_{\overline{p}}, \Delta_{\overline{p}})$ is klt.
%\end{proof}

\subsection{3-fold klt singularities are $\mathbb{F}_p$-rational}

We now show the $\mathbb{F}_p$-rationality of 3-fold singularities of klt type over an $F$-finite field. Our discussion includes the case of low characteristics $p=2,3$ and $5$ (for which rationality is known to fail in general \cites{CT19, Ber21, ABL22}) and the case of imperfect fields.
We first verify \autoref{conjecture} \autoref{conj: vanishing_nilpotent} in dimension 2.

\begin{proposition} \label{prop: nilp_vanishing_dP}
    Let $S$ be a surface over an $F$-finite field $k$, and let $f \colon S \to T$ be a morphism of Fano type.
    Then $R^i f_*\mathcal{O}_S \sim_F 0$ for $i>0$.
\end{proposition}

\begin{proof}
    If $\dim(T) \geq 1$, we conclude by the relative Kawamata--Viehweg vanishing for excellent surfaces \cite{Tan18}*{Theorem 3.3}.

    If $\dim(T)=0$, we have to show that $H^i(S, \cO_S) \sim_F 0$ for any $i > 0$. For $i \geq 2$, we clearly have $H^i(S, \cO_S) = 0$, so we only have to consider the case $i = 1$. We only have to show that $H^1(S_{\overline{k}}, \cO_{S_{\overline{k}}}) \sim_F 0$, and since $(S_{\overline{k}})_{\red}^{\nu} \to S_{\overline{k}}$ is a universal homeomorphism (see \cite{Tan-invariants-imperfect}*{Theorem 5.9} and \stacksproj{0CNF}), we deduce from \autoref{universal_homeomorphism_induces_isomorphism_of_crystals} that it is enough to show that \[ H^1\left((S_{\overline{k}})_{\red}^{\nu}, \cO_{(S_{\overline{k}})_{\red}^{\nu}}\right) = 0. \]  Let $\pi \colon T \to (S_{\overline{k}})_{\red}$ be a resolution of singularities. Since $\pi^* \colon H^1\left((S_{\overline{k}})_{\red}^{\nu}, \cO_{(S_{\overline{k}})_{\red}^{\nu}}\right) \to H^1(T, \cO_T)$ is an injection, it is enough to prove that $H^1(T, \cO_T) = 0$. This follows immediately from the fact that $T$ is smooth and rational by \cite{NT20}*{Proposition 2.26}.
\end{proof}

We now verify \autoref{conjecture} \autoref{conj: F_p-rationality} in dimension at most 3. The surface case follows from the fact that klt surface singularities are rational (see, for example, \cite[Proposition 2.28]{kk-singbook}), and thus $\mathbb{F}_p$-rational. Therefore it remains to consider the 3-dimensional case.

\begin{theorem}  \label{prop: 3-fold_sing}
Let $X$ be a 3-fold of klt type over an $F$-finite field. Then $X$ has $\mathbb{F}_p$-rational singularities.
\end{theorem}

\begin{proof}
    We only have to verify the hypotheses of \autoref{thm: F_p-rationality-Fano}: 
    the existence of resolution of singularities is established in \cite{CP19}, the birational MMP for 3-folds in \cite{k-notQfact}*{Theorem 9}, and the required vanishing for surfaces in \autoref{prop: nilp_vanishing_dP}.
\end{proof}

\section{Counterexample to Frobenius--stable GR vanishing theorem}

In this section, we construct wild quotient 3-fold singularities giving counterexamples to the Frobenius--stable version of the GR vanishing theorem. We first generalise  \cite{Forgarty_On_the_depth_of_local_rings_of_invariants_of_cyclic_groups}*{Proposition 4} to our setting.

\begin{lemma} \label{lem: loc_coh_quotient}
    Let $(R, \fm)$ be a Noetherian, $F$-finite local ring, such that $\dim(R) \geq 2$ and $\depth_{\crys}(R) \geq 2$. Let $G$ be a cyclic group of order divisible by $p$ acting on $R$ such that 
    \begin{itemize}
        \item this action is free on $\Spec R \setminus \{\fm\}$;
        \item $G$ acts trivially on the residue field $R/\fm$;
        \item $R^G$ is Noetherian, and $R^G \inc R$ is finite.
    \end{itemize}
    Then $\depth_{\crys}(R^G) \leq 2$.
\end{lemma}
\begin{proof}
    Let $X \coloneqq \Spec R$, let $\dot X \coloneqq X \setminus \{m\}$, let $Y \coloneqq X/G$, $\dot Y \coloneqq \dot X/G$, and let $\pi \colon X \to Y$ denote the quotient map. 
    The induced map $\pi \colon \dot X \to \dot Y$ is finite and étale (see \cite{Mumford_Abelian_Varieties}*{Proposition 2, page 70}). 

    Note that the following compositions of functors
    \[ \begin{tikzcd}
        \QCoh_G(\cO_{\dot X}) \arrow[rr, "(\cdot)^G"] &  & \QCoh(\cO_{\dot Y}) \arrow[rr, "{\Gamma(\dot Y, -)}"] &  & \Mod(\bF_p)
    \end{tikzcd} \] and 
    \[ \begin{tikzcd}
    \QCoh_G(\cO_{\dot X}) \arrow[rr, "{\Gamma(\dot X, -)}"] &  & \Mod_G(\bF_p) \arrow[rr, "(\cdot)^G"] &  & \Mod(\bF_p)
    \end{tikzcd} \] are equal.
    Furthermore, note that both $(\cdot)^G$ and $\Gamma(\dot X, -)$ preserve injective objects. Indeed, $(\cdot)^G$ is an equivalence of categories by \cite{Mumford_Abelian_Varieties} (they only treat the coherent case, but the quasi-coherent case is immediately deduced by taking unions), and $\Gamma(\dot X, -)$ has an exact left adjoint (the usual pullback).
    Therefore, we have a spectral sequence 
    \[ H^i\left(G, H^j\left(\dot X, \cO_{\dot X}^{1/p^\infty}\right)\right) \implies H^{i + j}\left(\dot Y, \cO_{\dot Y}^{1/p^\infty}\right). \] Since $\dot X$ has property $(S_2)$ up to nilpotence, we have $H^0(\dot X, \cO_{\dot X}^{1/p^{\infty}}) = R^{1/p^{\infty}}$. Thus, the exact sequence of terms in low degrees induces an injection \[ H^1(G, R^{1/p^{\infty}}) \hookrightarrow H^1\left(\dot Y, \sO_{\dot Y}^{1/p^\infty}\right). \] 
    We now finish the proof by assuming that $H^1(G, R^{1/p^{\infty}})\neq 0$. We note that $H^1(\dot Y, \sO_{\dot Y})\sim_{F}0$ if and only if $\varinjlim_j H^1(\dot Y, \sO_{\dot Y}^{1/p^j})=H^1(\dot Y, \varinjlim_j\sO_{\dot Y}^{1/p^j})=H^1(\dot Y, \sO_{\dot Y}^{1/p^{\infty}})=0$.
    Thus, we now have $H^1(\dot Y, \sO_{\dot Y})\not\sim_{F}0$.
    Since $Y$ is affine, we have $H^2_{\fm}(\sO_{Y})= H^1(\dot Y, \sO_{\dot Y})\not\sim_{F}0$, and we conclude.

    We now show that $H^1(G, R^{1/p^{\infty}}) \neq 0$.
    Let $\sigma$ be a generator of $G$, and let $n$ be the order of $G$. 
    We define a map $\Tr\colon R^{1/p^{\infty}}\to R^{1/p^{\infty}}$ by $\Tr(r) \coloneqq \sum_{i = 0}^{n - 1} \sigma^i(r)$ for $r\in R^{1/p^{\infty}}$.
    Using the 2-periodic resolution for cyclic groups of the trivial $\mathbb{Z}[G]$-module \cite{Bro94}*{I.(6.3)} we know that \[ 
    H^1\left(G, R^{1/p^\infty}\right) = \left\{r \in R^{1/p^\infty} \: | \: \Tr(r) = 0 \right\}/ \langle \sigma(r) - r \: | \: r \in R^{1/p^\infty}  \rangle. \] 

    We show that $1\in H^1\left(G, R^{1/p^\infty}\right)$ is a non-zero element. Firstly, $1\in \left\{ r\in R^{1/p^\infty} \: | \: \Tr(r) = 0 \right\} $ since $\Tr(1) = n=0$ by the assumption that $n$ is divisible by $p$. Suppose by contradiction that $1=0$ in $ H^1\left(G, R^{1/p^\infty}\right)$, i.e., there exist $r_1,\ldots, r_s \in R^{1/p^\infty}$ such that \[1=\sum_{i=1}^s (\sigma(r_i) - r_i).\]
    We take $N\in\Z_{>0}$ such that $r_i^{p^N}\in R$ for all $i$. Then taking $p^N$-th power, we obtain
    \[1=\sum_{i=1}^s \left(\sigma(r_i^{p^N}) - r_i^{p^N} \right).\] Since $G$ acts trivially on $R/\fm$, it follows that $\sigma(r_i^{p^N}) - r_i^{p^N}\in \fm$ for all $i=1, \dots, s$.
    Therefore, we obtain $1=\sum_i(\sigma(r_i^{p^N}) - r_i^{p^N})\in\fm$, a contradiction.
\end{proof}

We can now prove our main theorem (\autoref{thm: GR_fails_intro}).

\begin{theorem} \label{thm: counterex}
For any $p > 0$, there exists a $\mathbb{Q}$-factorial threefold $X$ over $\bF_p$ such that \[ R^1f_*\omega_Y \not\sim_C 0 \] for every resolution $f \colon Y \to X$. In addition, $X$ is $\bF_p$--rational and not Cohen--Macaulay up to nilpotence.

If $p \leq 5$, then we can take $X$ to have terminal singularities.
\end{theorem}

\begin{proof}
    Let us first deal with the case $p \leq 5$. Let $X$ be the terminal 3-fold constructed in \cite{Tot19}*{Theorem 5.1} (for $p = 2$), or in \cite{Tot24}*{Theorems 6.3 and 8.1} (respectively for $p = 3$ and $5$). In each of these cases, $X$ is a quotient of a smooth 3-fold by a $\bZ/p\bZ$-action, which is free in codimension 2. It is then automatically $\bQ$--factorial by \cite{km-book}*{Lemma 5.16}, and not CM up to nilpotence by \autoref{lem: loc_coh_quotient}. Furthermore, it has terminal singularities, and therefore has $\bF_p$--rational singularities by \autoref{prop: 3-fold_sing}. We then conclude that GR vanishing up to nilpotence must fail by \autoref{lem: equiv_conj_GR_CM}. \\

    Now, let $p > 0$ be any prime number, and let us generalise the example in \cite[Section 5]{Tot19}. Let $G = \mathbb{Z}/p\bZ = \langle \sigma \rangle$, and consider the $G$--action on $Y_0 \coloneqq (\mathbb{A}^1 \setminus \left\{1, 2,..., p-1 \right\})^3$ given by

    $$ \sigma (x_0, x_1, x_2) = \left(\frac{x_0}{1+x_0}, \frac{x_1}{1+x_1}, \frac{x_2}{1+x_2} \right). $$
    
    We automatically know by \autoref{lem: loc_coh_quotient} that $X_0 \coloneqq Y_0/G$ is not Cohen--Macaulay up to nilpotence. By \autoref{lem: equiv_conj_GR_CM}, we are left to show that $X_0$ has $\bF_p$--rational singularities (because then any resolution of singularities would then fail GR up to nilpotence). 
    
    Let $Y_1$ denote the blow--up of $Y_0$ at the origin, with exceptional divisor $F = \mathbb{P}^2$. The induced $G$--action is then described as

    $$ \sigma ((x_0, x_1, x_2), [w_0: w_1: w_2]) = \left( \left(\frac{x_0}{1+x_0}, \frac{x_1}{1+x_1}, \frac{x_2}{1+x_2} \right), \left[\frac{w_0}{1+x_0}: \frac{w_1}{1+x_1}: \frac{w_2}{1+x_2} \right] \right). $$

    We claim that the quotient $X_1 \coloneqq Y_1/G$ has $\mu_p$-quotient singularities, and hence strongly $F$-regular singularities by \cite[Theorem 4.3]{Pos24}.
    By symmetry, we only need to verify this in the chart $w_0 = 1$. In this chart, the action can then be written as 
    $$ \sigma (x_0, w_1, w_2) = \left(\frac{x_0}{1+x_0}, \frac{w_1(1+x_0)}{1+x_0w_1}, \frac{w_2(1+x_0)}{1+x_0w_2} \right), $$ where we recall that $x_0, x_0w_1, x_0w_2 \neq 1, \dots, p - 1$, since $x_1 = x_0w_1$ and $x_2 = x_0w_2$.

    Given $s \in \cO_{Y_1}$, we set $I(s) \coloneqq \sigma(s) - s$. The ideal of the fixed scheme of the action on $Y_1$ is then given by $(I(x_0), I(w_1), I(w_2))$. We compute that 
    
    $$I(x_0)=\frac{-x_0^2}{1+x_0}, \: I(w_1)=\frac{-x_0w_1(w_1-1)}{1+x_0w_1}, \: I(w_2)=\frac{-x_0w_2(w_2-1)}{1+x_0w_2}.$$
    Given that $1 + x_0$, $1 + x_0w_1$ and $1 + x_0w_2$ are units in our chart, we deduce that the ideal of the fixed scheme is given by $$(x_0^2, x_0w_1(w_1-1), x_0w_2(w_2-1)), $$ which is principal outside of the points $(0, 0, 0), (0, 1, 0), (0, 0, 1) $ and $(0, 1, 1).$ By \cite{KL13}*{Theorem 2}, we deduce that $X_1$ is regular outside of these points (in the chart $w_0 = 1$).

    To verify that these singularities are $ \mu_p$--quotients, we apply the criterion of \cite[Theorem 2.2]{Tot24}. Following the notation in \emph{loc. cit.}, let us set $e=s=x_0$. Given that $I(w_0) = I(w_0 - 1)$ and that $I(w_1) = I(w_1 - 1)$, our computation above shows that the hypotheses to apply \cite[Theorem 2.2]{Tot24} are satisfied, whence $X_1$ has $\mu_p$-quotient singularities.

    Fix a resolution of singularities $\pi \colon W \to X_1$, and let $f \colon X_1 \to X_0$ denote the birational map induced by the $G$--equivariant morphism $Y_1 \to Y_0$. Note that $X_1$ has in fact toric singularities (this is part of \cite{Tot24}*{Theorem 4.1}), so we know by \cite{Fulton_Toric}*{Proposition in p.76} that $R\pi_*\mathcal{O}_W=\mathcal{O}_{X_1}$. To deduce that $X_0$ has $\bF_p$--rational singularities, it is then enough to show that $Rf_*\mathcal{O}_{X_1} \sim_{F} \mathcal{O}_{X_0}$. 
    
    Let $E$ denote the exceptional divisor of $f$ and $F$ the exceptional divisor of $Y_1 \to Y_0$. By proper base change for Frobenius crystals (\autoref{proper_base_change_F_crystals}), it is enough to show that $H^i(E, \cO_E) \sim_F 0$ for all $i > 0$. By \autoref{universal_homeomorphism_induces_isomorphism_of_crystals}, it is then enough to show that the morphism $\bP^2 = F \to E$ induced by the quotient map is a universal homeomorphism. We may verify this after a base change to $\overline{\bF}_p$. Note that $F \to E$ remains bijective after the base change, since $F$ supports the fixed scheme by the previous local computation and group quotients commute with base field extensions. Note that this morphism is in particular universally injective by the equivalence between (2) and (4) in \stacksproj{01S4} (we use the fact that we work with schemes of finite type over an algebraically closed field here, so that surjectivity can be checked at rational points). We then conclude the proof by \stacksproj{04DF}. \qedhere
    \iffalse
    
    the induced extension of fraction fields $K(E) \subset K(\mathbb{P}^2)$ is purely inseparable. By \stacksproj{0BRA}, we only need to show that $E$ is normal. 
    
    Since $E$ is a $\bQ$--Cartier divisor (as $Y_1$ is $\bQ$--factorial) on a strongly $F$--regular variety, is it $S_2$ by \cite[Corollary 3.3]{PS14}. Letting $g \colon Y_1 \to X_1$ denote the quotient map, we know by the computation in \cite[pages 767-768]{Tot19} that $g^*E = F$. Given that $F$ is regular (hence locally cut by a regular parameter in $Y_1$), we deduce that this is also the case for $E$ (on the regular locus of $X_1$). In particular, $E$ is regular on the regular locus of $X_1$, so it is $R_1$. Thus, $E$ is indeed normal by Serre's criterion and the proof is complete.
    \fi
\end{proof}

\begin{remark}\label{rem: char_2_even_worse}
    When $p = 2$, there is even a $\bQ$-factorial $3$-fold $X$ over $\bF_2$ such that
    \begin{enumerate}
    \item $K_X$ is Cartier,
    \item $X$ is $F$-pure and has terminal singularities,
    \item $R^1f_*\omega_Y \not\sim_C 0$ for every resolution $f \colon Y \to X$. 
\end{enumerate}

Let us verify it. The wild quotient singularity of \cite{Tot19} is the quotient of $\mathbb{G}_m^3$ by the action $\mathbb{Z}/2\mathbb{Z}$ given by $(x, y, z) \to (x^{-1}, y^{-1}, z^{-1})$. Moreover, as mentioned in the proof of \cite{Tot19}*{Theorem 5.1}, $K_X$ is Cartier since the top-form $\omega=\frac{dx}{x} \wedge \frac{dy}{y} \wedge \frac{dz}{z}$ is invariant under the $\mathbb{Z}/2\mathbb{Z}$-action.
   
Since $K_X$ is Cartier, the $F$-purity of $X$ is equivalent to the surjectivity of the Frobenius trace map $C_X \colon F_*\omega_X \to \omega_X$ (cf.~\cite{fbook}*{Theorem 1.3.8}). As $C_{\mathbb{G}_m^3}(\omega)=\omega$ and $\omega$ descends on $X$, we deduce that $C_X$ is surjective as well.

Note that in this case, we can show that $X$ is not CM up to nilpotence without relying on our enhanced version of Fogarty's result. Indeed, since $X$ is $F$-pure, the Frobenius action on $H^i_{\m_x}(\sO_{X,x})$ is injective for all $i\geq 0$. Thus, Cohen-Macaulayness up to nilpotence is equivalent to usual Cohen--Macaulayness on $X$. However, we know that $X$ cannot be Cohen--Macaulay by \cite{Forgarty_On_the_depth_of_local_rings_of_invariants_of_cyclic_groups}*{Proposition 4}. 
    
\end{remark}

\begin{remark}\label{natural_compactification}
    There is a natural compactification of our affine example above for arbitrary $p > 0$. Namely, it is given the quotient of $\bP^1 \times \bP^1 \times \bP^1$ by the $\bZ/p\bZ$--action given by \[ ([x_0 : y_0], [x_1 : y_1], [x_2 : y_2]) \mapsto ([x_0 + y_0 : y_0],[x_1 + y_1 : y_1],[x_2 + y_2 : y_2]). \]
\end{remark}

We have therefore obtained a counterexample to a question of Hacon--Patakfalvi \cite{HP22}*{Question 1.6}.

\begin{corollary}\label{cor: counterex2}
    Over an algebraically closed field of characteristic $p > 0$, there exists a generically finite morphism $\pi \colon Y \to A$, where $Y$ and $A$ are smooth projective 3-folds, $A$ is an abelian variety, and $R^1\pi_*\omega_Y \not\sim_C 0$.
\end{corollary}

\begin{proof}
    The 3-fold singularity of \autoref{thm: counterex} can be used to obtain the desired counterexample following the construction in the proof of \cite{HK15}*{Proposition 3.13}. 
    
    One could also construct a more explicit example as follows. If $G \coloneqq \bZ/p\bZ$, consider the natural map $X_0 \coloneqq (\bP^1)^3/G \to (\bP^1/G)^3 \cong (\bP^1)^3$ (see \autoref{natural_compactification}), and take an abelian threefold $A$ that admits a finite morphism to $(\bP^1)^3$ which is étale at the image of the singularity $x_0$ of $X_0$. We can then consider $X_0'$ to be any irreducible component of $X_0 \times_{(\bP^1)^3} A$ (one can in fact show that this fiber product is connected and normal, but we do not need it here). If $Y$ denotes a resolution of singularities of $X_0'$, then the induced morphism $Y \to X_0'$ fails GR vanishing up to nilpotence by \autoref{doesnt_depend} (the morphism $X_0' \to X_0$ is étale around the singularity $x_0 \in X_0$, which itself fails GR up to nilpotence). Taking the composition $\pi \colon Y \to X_0' \to A$, we deduce that $R^1\pi_*\omega_Y \not \sim_{C} 0$ by the degeneration of the Leray spectral sequence for higher direct images (recall that $X_0' \to A$ is finite).
\end{proof}

Bhatt showed in \cite{Bha12}*{Proposition 6.3} that it is possible to kill any cohomology class of anti-big and anti-semi-ample line bundle by passing to a finite cover, whose degree must then necessarily be divisible by $p$. We now show that it is not possible to kill the cohomology classes just by taking purely inseparable covers, even in the smooth case.

\begin{corollary}\label{cor: counter3}
   Over an algebraically closed field of characteristic $p  > 0$, there exists a smooth projective 3-fold $Y$, a big and semiample line bundle $L$ on $Y$, such that for any purely inseparable finite cover $f \colon Z \to Y$, the induced morphism \[ H^2(Y, L^{-1}) \to H^2(Z, f^*L^{-1}) \] is non-zero.
\end{corollary}

\begin{proof}
    Let $\overline{X}$ be a projective 3-fold with only an isolated singularity $x \in \overline{X}$ as in \autoref{thm: counterex} (see e.g. \autoref{natural_compactification}), and let $\pi \colon Y \to \overline{X}$ be a resolution of singularities. Furthermore, let $D$ be an effective ample divisor on $\overline{X}$ with $x \notin D$, and let $L \coloneqq \pi^*\cO_{\overline{X}}(D)$. Since $\cO_{\overline{X}}(-D)$ is an ideal sheaf, it acquires the structure of a Frobenius module by \autoref{ex: F-mod-struct}, so also its pullback $L^{-1}$ has a natural Frobenius module structure.
    
    We claim that it is enough to show that $H^2(Y, L^{-1}) \not\sim_F 0$. Indeed, assume by contradiction that there exists a purely inseparable cover $f \colon Z \to Y$ such that the induced map on cohomology is zero. 
    We may very well assume that $f = F^s$. By definition, the Frobenius module structure on $L^{-1}$ factors as \[ L^{-1} \to F^s_*(F^{s, *}L^{-1}) \to F^s_*L^{-1}, \] where the second map comes from the multiplication by the section defining $(p^s - 1)\pi^*D$. The claim then follows.

    We now show that $H^2(Y, L^{-1}) \not\sim_F 0$. Note that if $\cM$ is a Frobenius module, then the structural morphism $\tau \colon \cM \to F_*\cM$ corresponds by adjunction to a morphism $\tau^* \colon F^*\cM \to \cM$. This gives us a way to define a Frobenius module structure on $F^*\cM$ (via the morphism $F^*\cM \to F_*F^{*}\cM$ corresponding to $F^*(\tau^*)$) and a morphism of Frobenius modules $F^*\cM \to \cM$ (which is simply $\tau^*)$. The kernel and cokernel of this morphism are by construction nilpotent, so $\cM \sim_F F^*\cM$. In our case, it is therefore enough to show that $H^2(Y, L^{-p^N}) \nsim_F 0$ for some $N \geq 0$ (with the Frobenius structure described above).

    We have by \autoref{thm: duality} and the projection formula that \[ R \Gamma(Y, L^{-1})^{\vee}  \sim_C R\Gamma(Y, L^{-p^N})^{\vee}\cong R\Gamma(Y, \omega_Y \otimes L^{p^N}) \cong R\Gamma(\overline{X}, R\pi_*\omega_Y \otimes \cO_{\overline{X}}(p^ND)). \]
    Since $D$ is ample and $N$ is sufficiently large, we know by Serre vanishing  that \[ \mathcal{H}^1R\Gamma(\overline{X}, R\pi_*\omega_Y \otimes \cO_{\overline{X}}(p^ND)) \cong H^0(\overline{X}, R^1\pi_*\omega_Y \otimes \cO_{\overline{X}}(p^N D)). \] Given that $x \notin D$ and $R^1\pi_*\omega_Y$ is supported at $x$, we obtain that \[ R^1\pi_*\omega_Y \cong R^1\pi_*\omega_Y \otimes \cO_{\overline{X}}(p^ND) \] as Cartier modules. Since $R^1\pi_*\omega_Y \not\sim_C 0$ and it is supported at a closed point, we deduce that \[ H^2(Y, L^{-1})^{\vee} \sim_C H^0(\overline{X}, R^1\pi_*\omega_Y) \not\sim_C 0, \] so $H^2(Y, L^{-1}) \not\sim_F 0$.
\end{proof}

\section{$\mathbb{F}_p$-rationality of 3-fold Fano type morphisms} \label{s: vanishing3foflds}

In this section, we show that the action of the Frobenius morphism is nilpotent on the higher direct images of the structure sheaf on 3-dimensional morphisms of Fano type over a perfect field.

\subsection{The absolute case}

We start by proving the nilpotency of the Frobenius action on the cohomology of the structure sheaf of 3-folds of Fano type.
For this, we prove some general results on the cohomology of rationally chain connected varieties. Fix a field $k$ of characteristic $p > 0$.

\begin{proposition}\label{prop:RCC}
    Let $X$ be a geometrically connected smooth proper rationally chain connected variety over $k$.
   Then $H^1(X,\cO_X)\sim_F 0$.
\end{proposition}
\begin{proof}
    We may assume that $k$ is algebraically closed. Since $X$ is smooth and rationally chain connected, 
    an \'etale covering of $X$ must be of degree prime to $p$ by \cite{Chambert-Loir}*{Proposition 8.4}. Thus, we have $H^1_{\et}(X,\mathbb{F}_p)=0$. We then conclude by \autoref{thm:RH} (or in this special case, usual Artin--Schreier theory \cite{LeiFu}*{Theorem 7.2.11}).
\end{proof}
\begin{proposition}\label{prop:Esnault}
    Let $X$ be a geometrically connected smooth proper rationally chain connected variety over $k$.
    Suppose that $H^3(X,\cO_X)\sim_F 0$.
    Then $H^i(X, \cO_X)\sim_F 0$ for $i\in\{1,2\}$.
\end{proposition}
\begin{proof}
   We may assume that $k$ is algebraically closed. By \autoref{prop:RCC}, we have $H^1(X, \cO_X)\sim_F 0$, or equivalently,
   $H^1_{\et}(X, \bF_p) = 0$. 
   We show $H^2_{\et}(X, \bF_p) = 0$. 
   For all $i>0$, we have $H^i(X, W\cO_{X, \bQ}) = 0$ 
   by \cite{Esnault}*{Theorem 1.1}, and thus $H^i_{\et}(X, \bQ_p) = 0$ by \cite{PZ21}*{Corollary 3.18}.
   Therefore, there exists $N > 0$ such that $H^i_{\et}(X, \bZ_p)=\varprojlim H^i_{\et}(X, \bZ/p^j\bZ)$ is $p^N$-torsion for all $i>0$. 
   Here, we recall that $H^i_{\et}(X, \bZ_p)\coloneqq R^i(\varprojlim\circ\Gamma_{\et})(\bZ/p^j\bZ)=\varprojlim H^i_{\et}(X, \bZ/p^j\bZ)$ holds because $H^{i-1}_{\et}(X, \bZ/p^j\bZ)$ is finite, and thus the inverse system $\{H^{i-1}_{\et}(X, \bZ/p^j\bZ)\}_j$ satisfies the Mittag-Leffler condition for all $i\geq 0$ (see \cite{Milne(etalecohomology)}*{Corollary VI.2.8} and \cite{Jannsen}*{(3.1)}).
   We also have $H^l_{\et}(X, \bZ/p^{j}\bZ) = 0$ for $l \in \{1, 3\}$ and all $j>0$, since $H^l_{\et}(X, \bF_p) = 0$ for $l \in \{1, 3\}$.  

   For all $j >0$, we have the following short exact sequence:
   \[ 0 \to \bZ/p^j\bZ \xrightarrow{p^N} \bZ/p^{j + N}\bZ \to \bZ/p^N\bZ \to 0\] 
    of locally constant sheaves in the \'etale topology.
   Taking the long exact sequence, we have
   \begin{multline*}
       H^1_{\et}(X, \bZ/p^N\bZ)=0 \to H^2_{\et}(X, \bZ/p^j\bZ) \xrightarrow{p^N} H^2_{\et}(X, \bZ/p^{j + N}\bZ) \to H^2_{\et}(X, \bZ/p^N\bZ)\\ \to H^3_{\et}(X, \bZ/p^j\bZ)=0.
   \end{multline*} 
   Since $H^2_{\et}(X, \bZ/p^j\bZ)$ is Artinian for all $j>0$ \cite{Milne(etalecohomology)}*{Corollary VI.2.8}, the inverse system $\{H^2_{\et}(X, \bZ/p^j\bZ)\}_j$ satisfies the Mittag-Leffler condition.
   Thus, taking inverse limits, we have an exact sequence 
   \[ 0 \to H^2_{\et}(X, \bZ_p) \xrightarrow{p^N} H^2_{\et}(X, \bZ_p) \to H^2_{\et}(X, \bZ/p^N\bZ) \to 0.\] 
   Since $p^N$ acts injectively on $H^2_{\et}(X, \bZ_p)$ and annihilates it, we have $H^2_{\et}(X, \bZ_p)=0$. 
   Thus, we obtain $H^2_{\et}(X, \bZ/p^N\bZ) = 0$ by the above exact sequence again. Since $H^3_{\et}(X, \bZ/p^j\bZ) = 0$ for all $j>0$ as proven above, the  short exact sequence of \'etale sheaves
   \[ 0 \to \bZ/p^{N - 1}\bZ \to \bZ/p^N\bZ \to \bF_p \to 0\] 
   shows that $H^2_{\et}(X, \bF_p) = 0$. 
   By \autoref{thm:RH}, we conclude that  $H^2(X, \mathcal{O}_X) \sim_F 0$.
\end{proof}

\begin{lemma} \label{lem: RCC_log_res_klt_type}
    Assume that $k$ is a perfect field, and let $X$ be a rationally chain connected proper 3-fold over $k$ of klt type. 
    Then every log resolution of $X$ is rationally chain connected.
\end{lemma}
\begin{proof}
This is proven by the same argument of \cite{GNT19}*{Theorem 4.8} using the birational 
MMP for 3-folds in every characteristic (cf.~\cite{HW22}*{Theorem 1.1} and \cite{k-notQfact}*{Theorem 9}).
\end{proof}

\begin{theorem}\label{thm:threefold}
    Assume that $k$ is a perfect field,  
    and let $X$ be a rationally chain connected proper 3-fold over $k$ of klt type with $H^3(X, \cO_X)\sim_F 0$. 
    Then $H^i(X,\cO_X)\sim_F 0$ for all $i>0$.
\end{theorem}

\begin{proof}
    Let $f\colon Y\to X$ be a log resolution of $X$. 
    Then $Y$ is rationally chain connected by \autoref{lem: RCC_log_res_klt_type}.
    Since $R^if_{*}\cO_Y\sim_F 0$ for $i>0$ by \autoref{prop: 3-fold_sing}, we deduce from the Leray spectral sequence 
    that $H^i(X,\cO_X)\sim_F H^i(Y,\cO_Y)$ for all $i>0$.
    In particular, we have $H^3(Y,\cO_Y)\sim_F 0$ by assumption.
    Now, \autoref{prop:Esnault} shows that $H^i(X,\cO_X)\sim_FH^i(Y,\cO_Y)\sim_F 0$ for all $i>0$.
\end{proof}

\begin{proposition}\label{thm:klt Fano}
    Assume that $k$ is perfect, and let $X$ be a proper 3-fold of klt type.
    Suppose that one of the following holds.
    \begin{enumerate}
        \item\label{klt Fano (a)} $\rho(X)=1$, $X$ is $\mathbb{Q}$-factorial, $-K_X$ is ample and $k$ is uncountable.
        \item\label{klt Fano (b)} $p>3$ and $X$ is of Fano type.
    \end{enumerate}
    Then $H^i(X,\cO_X)\sim_F 0$ for all $i>0$.
\end{proposition}
\begin{proof}
    Since $-K_X$ is big, we have $H^3(X, \sO_X)\cong H^0(X, \sO_X(K_X))=0$ in both cases.
    Thus, it suffices to verity that $X$ is rationally chain connected by \autoref{thm:threefold}.
    In the case of \autoref{klt Fano (a)}, this follows from \cite{GLP$^+$15}*{Proposition 3.4 (2)}.
    In the case of \autoref{klt Fano (b)}, since $p>3$, we can argue as in the proof of \cite{GNT19}*{Proposition 4.11} to conclude that $X$ is rationally chain connected by the MMP in characteristic 5 \cite{HW-MMp-5}.
\end{proof}

We conclude with an observation on the Frobenius action on the cohomology of a rationally chain connected surface.

\begin{proposition}
    Assume that $k$ is perfect, and let $X$ be a normal projective rationally chain connected surface with $W\mathcal{O}$-rational singularities. 
    Then $H^i(X,\cO_X)\sim_F 0$ for all $i>0$.
\end{proposition}
\begin{proof}
    Let $f \colon Y \to X$ be a log resolution of singularities. 
    Since $X$ has only $W\mathcal{O}$-rational singularities, the exceptional locus $E$ is a tree of smooth rational curves by \cite{NT20}*{Proposition 2.23} (this is not true for $\F_p$-rational singularities, because the affine cone of a supersingular elliptic curve is $\F_p$-rational).
    This implies that $Y$ is rationally chain connected and that $R^if_*\cO_Y \sim_F 0$ for $i>0$.
    By \autoref{prop:Esnault}, we have $H^i(Y, \sO_Y) \sim_F 0$ for all $i>0$.
    Therefore, we conclude from the Leray spectral sequence that $H^i(X, \sO_X) \sim_F 0$ for all $i>0$.
\end{proof}

\subsection{The relative case}

In this section, we prove the vanishing for 3--dimensional morphisms of Fano type onto a positive-dimensional basis. 
To show this, we use the extraction of Koll\'{a}r components as in \cite{GNT19}*{Proposition 2.15} and the proper base change theorem for Frobenius crystals.

\begin{lemma}\label{lem: Fano-type}
    Let $(X, \Delta)$ be a klt $3$-fold pair over a perfect field $k$ of characteristic $p > 3$. Let $f \colon X \to Z$ be a Fano type contraction such that $\dim(Z) \geq 1$.
    Then $R^if_*\cO_X \sim_F 0$ for all $i>0$. 
\end{lemma}
\begin{proof}
%    By \autoref{thm:RH}, it is sufficient to prove that $\Sol(R^if_*\cO_X)=0$.
%    This is equivalent to $\Supp(\Sol(R^if_* \cO_X))=\emptyset$ and thus, by \cite{Bau23}*{Lemma 2.2.17},
%    it is enough to show that $(R^if_*\cO_X) \otimes k(z) \sim_F 0$ for all closed points $z \in Z$ and $i>0$
By \autoref{lem: vanishing fibers}, it is enough to show that $R^if_*\cO_X \otimes k(z) \sim_F 0$ for all closed points $z \in Z$ and $i>0$.

    Let $z \in Z$ be a closed point. By \cite{GNT19}*{Proposition 2.15} (although it only deals with the case $p > 5$, the case $p = 5$ can be proven with the MMP from \cite{HW-MMp-5}), we have a commutative diagram
    \[
    \begin{tikzcd}
        W \arrow[r, "\varphi"] \arrow[d, "\psi"'] & Y \arrow[d, "g"] \\
        X \arrow[r, "f"']                      & Z 
    \end{tikzcd} 
    \] 
    and an effective $\bQ$-divisor $\Delta_Y$ on $Y$ such that 
    \begin{itemize}
        \item $W$ is smooth,
        \item $\varphi, \psi$ are projective birational morphisms;
        \item $(Y, \Delta_Y)$ is plt and $\bQ$-factorial;
        \item $-(K_Y + \Delta_Y)$ is $g$-ample;
        \item $(Y_z)_{\red} = \lfloor \Delta_Y \rfloor$.
       % \item $-\lfloor \Delta_Y \rfloor$ is $g$-nef. \textcolor{red}{Do we use this condition?}
    \end{itemize} 
    By \autoref{prop: 3-fold_sing}, we have \[ R^if_*\cO_X \sim_F R^i(f \circ \psi)_*\cO_W \cong R^i(g \circ \varphi)_*\cO_W \sim_F R^ig_*\cO_Y, \] so that by the proper base change theorem (\autoref{proper_base_change_F_crystals}) we conclude \[ R^if_*\cO_X \otimes k(z) \sim_F  H^i(Y_z, \cO_{Y_z}). \]
    
    Consider the integral surface $S$ corresponding to the prime divisor $\lfloor \Delta_Y \rfloor$. 
    By \cite{HW22}*{Theorem 1.2}, $S$ is normal up to a universal homeomorphism, so in particular the composition $S^{\nu} \to S= (Y_z)_{\red}\to  Y_z$ is a universal homeomorphism. Thus, $H^i(Y_z, \cO_{Y_z}) \sim_F H^i(S^{\nu}, \cO_{S^{\nu}})$ by \autoref{universal_homeomorphism_induces_isomorphism_of_crystals}.  
    By adjunction, $S^{\nu}$ is a surface of del Pezzo type, so we conclude from \autoref{prop: nilp_vanishing_dP}. 
\end{proof}

\begin{remark}
    The same proof also works over $F$-finite fields of characteristic $p>5$, using the results in \cites{DW22, Wal23}.
\end{remark}

\begin{theorem}\label{thm: vanishing_3fold_Fanotype}
      Let $k$ be a perfect field of characteristic $p>3$, and let $f \colon X \to Z$ be a Fano type morphism over $k$ such that $\dim(X)=3$.
      Then $R^if_*\cO_X \sim_F 0$ for all $i>0$.
\end{theorem}

\begin{proof}
    If $\dim(Z)=0$ (resp.~$>0$), this is proven in \autoref{thm:klt Fano} (resp.~\autoref{lem: Fano-type}).
\end{proof}

We can now combine all previous results to show $\mathbb{F}_p$-rationality of klt 4-folds admitting a log resolution.

\begin{corollary} \label{cor: 4folds_klt}
    Let $X$ be a 4-fold of klt type over a perfect field $k$ of characteristic $p>5$.
    Assume that a log resolution exists for every birational model of $X$.
    Then $X$ has $\mathbb{F}_p$-rational singularities. 
\end{corollary}

\begin{proof}
    As the birational MMP is established in 
    \autoref{thm: kollar-notqfactorial-4folds}, and the vanishing up to Frobenius nilpotence for 3-folds of Fano type in \autoref{thm: vanishing_3fold_Fanotype}, we can conclude by \autoref{thm: F_p-rationality-Fano}.
\end{proof}

\section{Kawamata--Viehweg vanishing on $K$-trivial 3-folds}

We show that a Frobenius--stable version of the Kawamata--Viehweg theorem holds on klt Calabi--Yau pairs of dimension 3. We will need the following standard fact about triangulated categories:

\begin{lemma}\label{derived_cat_lemma}
    Let $\mathcal{T}$ be a triangulated category (e.g. a derived category), and consider a diagram 
    \[ \begin{tikzcd}
        & \mathcal{C}^{\bullet} \arrow[d, "a"] &                               &    \\
        \mathcal{F}^{\bullet} \arrow[r, "f"] & \mathcal{G}^{\bullet} \arrow[r, "g"] & \mathcal{H}^{\bullet} \arrow[r, "+1"] & {}
    \end{tikzcd} \] in $\mathcal{T}$, where the bottom row is an exact triangle. If $g \circ a = 0$, then $a$ factors through $\cF^{\bullet}$. 
\end{lemma}
\begin{proof}
    It follows from \cite[10.2.8]{Wei94}.
\end{proof}

We also need the following generalisation of $p$-power freeness proven by Tanaka \cite{tan-p-power}.

\begin{proposition} \label{prop: p-power-freeness}
    Let $k$ be an $F$-finite field.
    Let $f \colon X \to Z$ be a proper birational morphism of quasi-projective normal 3-folds over $k$.
    Assume that $Z$ is klt type. 
    If $L$ is Cartier divisor such that $L \equiv_{f} 0$, there exists $e>0$ such that $p^e L \sim f^*L_Z$, where $L_Z$ is Cartier.
\end{proposition}

\begin{proof}
    The same proof of \cite{tan-p-power}*{Theorem 4.3} works with our hypothesis.
    Indeed, the MMP needed is proved in \cite{7authors}*{Theorem 9.15} and the analogue of \cite{tan-p-power}*{Lemma 4.2} holds by \cite{Tan18}*{Theorem 4.1} and \cite{BT22}*{Theorem 1.3}. 
\end{proof}

\begin{theorem}\label{thm: kvv-up-frob}
    Let $k$ be an $F$-finite field field of characteristic $p>0$.
    Let $(X, \Delta)$ be a projective klt 3-fold pair over $k$ with $K_X+\Delta \sim_{\bQ} 0$.
    If $D$ is a big and semi-ample Cartier divisor on $X$, then for $e \gg 0$ the morphism \[F^e \colon H^i(X, \mathcal{O}_X(D)) \to H^i(X, F^e_{*}\mathcal{O}_X(p^eD))\] vanishes for every $i>0$.
\end{theorem}

\begin{remark}
    Note that $D \sim_{\bQ} (K_X + \Delta) + D$, so this theorem is really a Kawamata--Viehweg vanishing--type result up to an iterated Frobenius pullback.
\end{remark}

\begin{proof}
    % By \st{\cite[Theorem 4.3]{tan-p-power}} \footnote{The hypothesis that the characteristic $p$ is larger than 5 in \emph{op.cit.}   needed for the extraction of a Koll\'{a}r component, now guaranteed in small characteristic by \cite{k-notQfact}.}}, 
    Since $D$ is semiample and big, there is a birational contraction $f \colon X \to Z$ such that $D \equiv_f 0$. As $f_*(K_X+\Delta)=K_Z+f_*\Delta \sim_{\mathbb{Q}} 0$, we deduce that $K_X+\Delta-f^*(K_Z+f_*\Delta) \sim_{\mathbb{Q}} 0$, meaning that $(X, \Delta)$ is crepant to $(Z, f_*\Delta)$ and thus $Z$ has singularities of klt type. 
    By \autoref{prop: p-power-freeness} we then obtain that $\mathcal{O}_X(p^eD)$ is globally generated for $e \gg 0$ and $p^eD \sim f^*A$ for some very ample Cartier divisor $A$ on $Z$. Since we may replace $D$ by $p^eD$, we may assume that $D \sim f^*A$.
    %By the negativity lemma, $Z$ has still  singularities of klt type 
     Moreover, as $Z$ has klt type singularities, we have $R^if_{*}\mathcal{O}_X \sim_{F} 0$ for $i>0$ by \autoref{prop: 3-fold_sing}. 

    \begin{claim}
        There exists $e > 0$ such that the Frobenius action $Rf_*\cO_X \to F^e_*Rf_*\cO_X$ factors as \[ Rf_*\cO_X \to F^e_*f_*\cO_X \to F^e_*Rf_*\cO_X \] in $D^b(\Coh_Z)$.
    \end{claim}
    \noindent\emph{Proof of the claim.}
    Let $m \geq 0$ be the biggest integer such that $R^mf_*\cO_Y \neq 0$. If $m = 0$, then there is nothing to do, so assume that $m > 0$. For each $i \in \{1, \dots, m\}$, let $e_i > 0$ be such that the Frobenius action $R^if_*\cO_X \to F^{e_i}_*R^if_*\cO_X$ is zero. We have a commutative diagram of exact triangles in $D^b(\Coh_Z)$:

    \[ \begin{tikzcd}
        \tau_{\leq m - 1}Rf_*\cO_X \arrow[r] \arrow[d]  & Rf_*\cO_X \arrow[d, "a"] \arrow[r] & R^mf_*\cO_X \arrow[d]     \arrow[r, "+1"]           &   {} \\
        F^{e_m}_*(\tau_{\leq m - 1}Rf_*\cO_X) \arrow[r] & F^{e_m}_*Rf_*\cO_X \arrow[r]       & F^{e_m}_*R^mf_*\cO_X \arrow[r, "+1"] & {}
    \end{tikzcd} \] so by \autoref{derived_cat_lemma}, the map $Rf_*\cO_X \to F^{e_m}_*Rf_*\cO_X$ factors through $F^{e_m}_*(\tau_{\leq m - 1}Rf_*\cO_X)$. We can then keep going (replacing $Rf_*\cO_X$ by $\tau_{\leq m - 1}Rf_*\cO_X$ and so on) to conclude that for $e = e_m + \dots + e_1$, the map $Rf_*\cO_X \to F^e_*Rf_*\cO_X$ factors through $F^e_*f_*\cO_X$. Hence, we have proven the claim. \\ 

We can now twist by $\cO_Z(A)$ to obtain, by the projection formula, a factorisation \[ Rf_*\cO_X(D) \to F^e_*\cO_Z(p^eA) \to F^e_*Rf_*\cO_X(p^eD). \] Applying $H^i(Z, -)$ gives us the factorisation \[ H^i(X, \cO_X(D)) \to H^i(Z, F^e_*\cO_Z(p^eA)) \to H^i(X, F^e_*\cO_X(p^eD)), \] 
    so for $e \gg 0$, this composition is zero by Serre vanishing.
\end{proof}

\begin{corollary} 
    Let $k$ be an $F$-finite field of characteristic $p>0$ and let $(X, \Delta)$ be a projective klt 3-fold pair over $k$ with $K_X+\Delta \sim_{\bQ} 0$. Assume further that $X$ is globally $F$-split. Then for any big and semi-ample Cartier divisor $D$ on $X$ and $i > 0$, we have \[H^i(X, \mathcal{O}_X(D)) = 0. \]
\end{corollary}

\begin{proof}
    Since $X$ is globally $F$-split, $F^e \colon H^i(X, \mathcal{O}_X(D))\to H^i(X, F^e_*\mathcal{O}_X(p^eD))$ splits, so we conclude by \autoref{thm: kvv-up-frob}.
\end{proof}

\begin{remark}\label{rem: vanishing_CY_fourfolds}
    The exact same proof shows that if $(X, \Delta)$ is a projective klt Calabi--Yau pair of dimension 4 satisfying \autoref{hyp} over a perfect field $k$ of characteristic $p > 5$ and $D$ is a big and semiample Cartier divisor on $X$, then there exists $j > 0$ coprime to $p$ such that for all $e \gg 0$, the morphism \[ F^e \colon H^i(X, \cO_X(jD)) \to  H^i(X, F^e_*\cO_X(p^ejD)) \] vanishes for any $i > 0$. We do not know if we can take $j = 1$ because we do not have an analogue of \cite[Theorem 4.3]{tan-p-power} in dimension 4 as we do not know if a numerically trivial Cartier divisor on a klt Fano 3-fold is $p^{e}$-torsion for some $e>0$.
    
    In any case, we deduce that if $X$ is 
    furthermore globally $F$--split, then $H^i(X, \cO_X(jD)) = 0$ for all $i > 0$.
\end{remark}

\bibliographystyle{amsalpha}
\bibliography{refJAG}

% \bib, bibdiv, biblist are defined by the amsrefs package.
\begin{bibdiv}
\begin{biblist}

\bib{ABL22}{article}{
      author={Arvidsson, Emelie},
      author={Bernasconi, Fabio},
      author={Lacini, Justin},
       title={On the {K}awamata--{V}iehweg vanishing theorem for log del
  {P}ezzo surfaces in positive characteristic},
        date={2022},
     journal={Compos. Math.},
      volume={158},
      number={4},
       pages={750\ndash 763},
}

\bib{Bau23}{article}{
      author={Baudin, Jefferson},
       title={Duality between {C}artier crystals and perverse
  $\mathbb{F}_p$-sheaves, and application to generic vanishing},
        date={2023},
     journal={arXiv e-prints},
        note={Available at
  \href{https://arxiv.org/abs/2306.05378}{arXiv:2306.05378}},
}

\bib{BB11}{article}{
      author={Blickle, M.},
      author={B\"{o}ckle, G.},
       title={Cartier modules: finiteness results},
        date={2011},
     journal={J. Reine Angew. Math.},
      volume={661},
       pages={85\ndash 123},
}

\bib{Berthelot_Bloch_Esnault_On_Witt_vector_cohomology_for_singular_varieties}{article}{
      author={Berthelot, Pierre},
      author={Bloch, Spencer},
      author={Esnault, H\'el\`ene},
       title={On {W}itt vector cohomology for singular varieties},
        date={2007},
     journal={Compos. Math.},
      volume={143},
      number={2},
       pages={363\ndash 392},
}

\bib{BBLSZ23}{article}{
      author={Bhatt, Bhargav},
      author={Blickle, Manuel},
      author={Lyubeznik, Gennady},
      author={Singh, Anurag~K.},
      author={Zhang, Wenliang},
       title={Applications of perverse sheaves in commutative algebra},
        date={2025},
     journal={J. Reine Angew. Math.},
        note={published online},
}

\bib{BE08}{article}{
      author={Blickle, Manuel},
      author={Esnault, H\'{e}l\`ene},
       title={Rational singularities and rational points},
        date={2008},
     journal={Pure Appl. Math. Q.},
      volume={4},
      number={3, Special Issue: In honor of Fedor Bogomolov. Part 2},
       pages={729\ndash 741},
}

\bib{Ber21}{article}{
      author={Bernasconi, Fabio},
       title={Kawamata--{V}iehweg vanishing fails for log del {P}ezzo surfaces
  in characteristic 3},
        date={2021},
     journal={J. Pure Appl. Algebra},
      volume={225},
      number={11},
       pages={106727, 16},
}

\bib{Bha12}{article}{
      author={Bhatt, B.},
       title={Derived splinters in positive characteristic},
        date={2012},
     journal={Compos. Math.},
      volume={148},
      number={6},
       pages={1757\ndash 1786},
}

\bib{fbook}{book}{
      author={Brion, Michel},
      author={Kumar, Shrawan},
       title={Frobenius splitting methods in geometry and representation
  theory},
      series={Progress in Mathematics},
   publisher={Birkh\"{a}user Boston, Inc., Boston, MA},
        date={2005},
      volume={231},
        ISBN={0-8176-4191-2},
      review={\MR{2107324}},
}

\bib{BK23}{article}{
      author={Bernasconi, Fabio},
      author={Koll\'{a}r, J\'{a}nos},
       title={Vanishing {T}heorems for {T}hree-folds in {C}haracteristic
  {$p>5$}},
        date={2023},
     journal={Int. Math. Res. Not. IMRN},
      number={4},
       pages={2846\ndash 2866},
}

\bib{Bhatt_Lurie_RH_corr_pos_char}{article}{
      author={Bhatt, B.},
      author={Lurie, J.},
       title={A {R}iemann-{H}ilbert correspondence in positive characteristic},
        date={2019},
     journal={Camb. J. Math.},
      volume={7},
      number={1-2},
       pages={71\ndash 217},
}

\bib{BM76}{article}{
      author={Bombieri, E.},
      author={Mumford, D.},
       title={Enriques' classification of surfaces in char. {$p$}. {III}},
        date={1976},
     journal={Invent. Math.},
      volume={35},
       pages={197\ndash 232},
}

\bib{7authors}{article}{
      author={Bhatt, Bhargav},
      author={Ma, Linquan},
      author={Patakfalvi, {Zs}olt},
      author={Schwede, Karl},
      author={Tucker, Kevin},
      author={Waldron, Joe},
      author={Witaszek, Jakub},
       title={Globally +-regular varieties and the minimal model program for
  threefolds in mixed characteristic},
        date={2023},
     journal={Publ. Math. Inst. Hautes \'{E}tudes Sci.},
      volume={138},
       pages={69\ndash 227},
}

\bib{Bockle_Pink_Cohomological_Theory_of_crystals_over_function_fields}{book}{
      author={B\"{o}ckle, G.},
      author={Pink, R.},
       title={Cohomological theory of crystals over function fields},
      series={EMS Tracts in Mathematics},
   publisher={European Mathematical Society (EMS), Z\"{u}rich},
        date={2009},
      volume={9},
}

\bib{Bro94}{book}{
      author={Brown, K.~S.},
       title={Cohomology of groups},
      series={Graduate Texts in Mathematics},
   publisher={Springer-Verlag, New York},
        date={1994},
      volume={87},
        note={Corrected reprint of the 1982 original},
}

\bib{BS13}{incollection}{
      author={Blickle, Manuel},
      author={Schwede, Karl},
       title={{$p^{-1}$}-linear maps in algebra and geometry},
        date={2013},
   booktitle={Commutative algebra},
   publisher={Springer, New York},
       pages={123\ndash 205},
}

\bib{BT22}{article}{
      author={Bernasconi, Fabio},
      author={Tanaka, Hiromu},
       title={On del {P}ezzo fibrations in positive characteristic},
        date={2022},
     journal={J. Inst. Math. Jussieu},
      volume={21},
      number={1},
       pages={197\ndash 239},
}

\bib{Cartier_une_nouvelle_operation_sur_les_formes_differentielles}{article}{
      author={Cartier, P.},
       title={Une nouvelle op\'{e}ration sur les formes diff\'{e}rentielles},
        date={1957},
     journal={C. R. Acad. Sci. Paris},
      volume={244},
       pages={426\ndash 428},
}

\bib{Chambert-Loir}{article}{
      author={Chambert-Loir, Antoine},
       title={Points rationnels et groupes fondamentaux: applications de la
  cohomologie {$p$}-adique (d'apr\`es {P}. {B}erthelot, {T}. {E}kedahl, {H}.
  {E}snault, etc.)},
        date={2004},
     journal={Ast\'{e}risque},
      number={294},
       pages={viii, 125\ndash 146},
}

\bib{CP19}{article}{
      author={Cossart, Vincent},
      author={Piltant, Olivier},
       title={Resolution of singularities of arithmetical threefolds},
        date={2019},
     journal={J. Algebra},
      volume={529},
       pages={268\ndash 535},
}

\bib{CR11}{article}{
      author={Chatzistamatiou, Andre},
      author={R\"{u}lling, Kay},
       title={Higher direct images of the structure sheaf in positive
  characteristic},
        date={2011},
        ISSN={1937-0652},
     journal={Algebra Number Theory},
      volume={5},
      number={6},
       pages={693\ndash 775},
}

\bib{CR12}{article}{
      author={Chatzistamatiou, Andre},
      author={R\"{u}lling, Kay},
       title={Hodge-{W}itt cohomology and {W}itt-rational singularities},
        date={2012},
     journal={Doc. Math.},
      volume={17},
       pages={663\ndash 781},
}

\bib{CR15}{article}{
      author={Chatzistamatiou, Andre},
      author={R\"{u}lling, Kay},
       title={Vanishing of the higher direct images of the structure sheaf},
        date={2015},
     journal={Compos. Math.},
      volume={151},
      number={11},
       pages={2131\ndash 2144},
}

\bib{CT19}{article}{
      author={Cascini, Paolo},
      author={Tanaka, Hiromu},
       title={Purely log terminal threefolds with non-normal centres in
  characteristic two},
        date={2019},
     journal={Amer. J. Math.},
      volume={141},
      number={4},
       pages={941\ndash 979},
}

\bib{DW22}{article}{
      author={Das, Omprokash},
      author={Waldron, Joe},
       title={On the log minimal model program for threefolds over imperfect
  fields of characteristic {$p>5$}},
        date={2022},
     journal={J. Lond. Math. Soc. (2)},
      volume={106},
      number={4},
       pages={3895\ndash 3937},
}

\bib{Elk78}{article}{
      author={Elkik, Ren\'{e}e},
       title={Singularit\'{e}s rationnelles et d\'{e}formations},
        date={1978},
     journal={Invent. Math.},
      volume={47},
      number={2},
       pages={139\ndash 147},
}

\bib{Elk81}{article}{
      author={Elkik, Ren\'{e}e},
       title={Rationalit\'{e} des singularit\'{e}s canoniques},
        date={1981},
     journal={Invent. Math.},
      volume={64},
      number={1},
       pages={1\ndash 6},
}

\bib{Esnault}{article}{
      author={Esnault, H\'{e}l\`ene},
       title={Varieties over a finite field with trivial {C}how group of
  0-cycles have a rational point},
        date={2003},
     journal={Invent. Math.},
      volume={151},
      number={1},
       pages={187\ndash 191},
}

\bib{EV85}{article}{
      author={Esnault, H.},
      author={Viehweg, E.},
       title={Two-dimensional quotient singularities deform to quotient
  singularities},
        date={1985},
     journal={Math. Ann.},
      volume={271},
      number={3},
       pages={439\ndash 449},
}

\bib{Forgarty_On_the_depth_of_local_rings_of_invariants_of_cyclic_groups}{article}{
      author={Fogarty, J.},
       title={On the depth of local rings of invariants of cyclic groups},
        date={1981},
     journal={Proc. Amer. Math. Soc.},
      volume={83},
      number={3},
       pages={448\ndash 452},
}

\bib{LeiFu}{book}{
      author={Fu, Lei},
       title={Etale cohomology theory},
      series={Nankai Tracts in Mathematics},
   publisher={World Scientific Publishing Co. Pte. Ltd., Hackensack, NJ},
        date={2011},
      volume={13},
}

\bib{Fulton_Toric}{book}{
      author={Fulton, William},
       title={Introduction to toric varieties},
      series={Annals of Mathematics Studies},
   publisher={Princeton University Press, Princeton, NJ},
        date={1993},
      volume={131},
        ISBN={0-691-00049-2},
         url={https://doi.org/10.1515/9781400882526},
        note={The William H. Roever Lectures in Geometry},
      review={\MR{1234037}},
}

\bib{GLP$^+$15}{article}{
      author={Gongyo, Yoshinori},
      author={Li, Zhiyuan},
      author={Patakfalvi, {Zs}olt},
      author={Schwede, Karl},
      author={Tanaka, Hiromu},
      author={Zong, Runhong},
       title={On rational connectedness of globally {$F$}-regular threefolds},
        date={2015},
     journal={Adv. Math.},
      volume={280},
       pages={47\ndash 78},
}

\bib{GNT19}{article}{
      author={Gongyo, Yoshinori},
      author={Nakamura, Yusuke},
      author={Tanaka, Hiromu},
       title={Rational points on log {F}ano threefolds over a finite field},
        date={2019},
        ISSN={1435-9855},
     journal={J. Eur. Math. Soc. (JEMS)},
      volume={21},
      number={12},
       pages={3759\ndash 3795},
}

\bib{GR70}{article}{
      author={Grauert, Hans},
      author={Riemenschneider, Oswald},
       title={Verschwindungss\"{a}tze f\"{u}r analytische {K}ohomologiegruppen
  auf komplexen {R}\"{a}umen},
        date={1970},
     journal={Invent. Math.},
      volume={11},
       pages={263\ndash 292},
}

\bib{HK15}{incollection}{
      author={Hacon, C.~D.},
      author={Kov\'{a}cs, S.~J.},
       title={Generic vanishing fails for singular varieties and in
  characteristic {$p>0$}},
        date={2015},
   booktitle={Recent advances in algebraic geometry},
      series={London Math. Soc. Lecture Note Ser.},
      volume={417},
   publisher={Cambridge Univ. Press, Cambridge},
       pages={240\ndash 253},
}

\bib{HP22}{article}{
      author={Hacon, Christopher~D.},
      author={Patakfalvi, {Zs}olt},
       title={Generic vanishing in characteristic {$p>0$} and the geometry of
  theta divisors},
        date={2022},
     journal={Boll. Unione Mat. Ital.},
      volume={15},
      number={1-2},
       pages={215\ndash 244},
}

\bib{HS77}{article}{
      author={Hartshorne, Robin},
      author={Speiser, Robert},
       title={Local cohomological dimension in characteristic {$p$}},
        date={1977},
     journal={Ann. of Math. (2)},
      volume={105},
      number={1},
       pages={45\ndash 79},
}

\bib{HW19}{article}{
      author={Hacon, Christopher~D.},
      author={Witaszek, Jakub},
       title={{On the rationality of {K}awamata log terminal singularities in
  positive characteristic}},
        date={2019},
        ISSN={2313-1691},
     journal={Algebr. Geom.},
      volume={6},
      number={5},
       pages={516\ndash 529},
}

\bib{HW-MMp-5}{article}{
      author={Hacon, Christopher~D.},
      author={Witaszek, Jakub},
       title={The minimal model program for threefolds in characteristic 5},
        date={2022},
        ISSN={0012-7094},
     journal={Duke Math. J.},
      volume={171},
      number={11},
       pages={2193\ndash 2231},
}

\bib{HW22}{article}{
      author={Hacon, Christopher~D.},
      author={Witaszek, Jakub},
       title={On the relative minimal model program for threefolds in low
  characteristics},
        date={2022},
     journal={Peking Math. J.},
      volume={5},
      number={2},
       pages={365\ndash 382},
}

\bib{HW20}{article}{
      author={Hacon, Christopher~D.},
      author={Witaszek, Jakub},
       title={On the relative {M}inimal {M}odel {P}rogram for fourfolds in
  positive and mixed characteristic},
        date={2023},
     journal={Forum of Mathematics, Pi},
      volume={11},
      number={e10},
       pages={1\ndash 35},
}

\bib{HX15}{article}{
      author={Hacon, Christopher~D.},
      author={Xu, Chenyang},
       title={On the three dimensional minimal model program in positive
  characteristic},
        date={2015},
     journal={J. Amer. Math. Soc.},
      volume={28},
      number={3},
       pages={711\ndash 744},
}

\bib{Ill79}{article}{
      author={Illusie, Luc},
       title={Complexe de de {R}ham-{W}itt et cohomologie cristalline},
        date={1979},
     journal={Ann. Sci. \'{E}cole Norm. Sup. (4)},
      volume={12},
      number={4},
       pages={501\ndash 661},
}

\bib{Jannsen}{article}{
      author={Jannsen, Uwe},
       title={Continuous \'{e}tale cohomology},
        date={1988},
     journal={Math. Ann.},
      volume={280},
      number={2},
       pages={207\ndash 245},
}

\bib{Kaw21}{article}{
      author={Kawakami, Tatsuro},
       title={On {K}awamata-{V}iehweg type vanishing for three dimensional
  {M}ori fiber spaces in positive characteristic},
        date={2021},
     journal={Trans. Amer. Math. Soc.},
      volume={374},
      number={8},
       pages={5697\ndash 5717},
}

\bib{KL13}{article}{
      author={Kir\`aly, Franz},
      author={L\"{u}tkebohmert, Werner},
       title={Group actions of prime order on local normal rings},
        date={2013},
        ISSN={1937-0652},
     journal={Algebra Number Theory},
      volume={7},
      number={1},
       pages={63\ndash 74},
         url={https://doi.org/10.2140/ant.2013.7.63},
      review={\MR{3037890}},
}

\bib{km-book}{book}{
      author={Koll\'{a}r, J\'{a}nos},
      author={Mori, Shigefumi},
       title={Birational geometry of algebraic varieties},
      series={Cambridge Tracts in Mathematics},
   publisher={Cambridge University Press, Cambridge},
        date={1998},
      volume={134},
        note={With the collaboration of C. H. Clemens and A. Corti, Translated
  from the 1998 Japanese original},
}

\bib{kk-singbook}{book}{
      author={Koll\'{a}r, J\'{a}nos},
       title={Singularities of the minimal model program},
      series={Cambridge Tracts in Mathematics},
   publisher={Cambridge University Press, Cambridge},
        date={2013},
      volume={200},
        ISBN={978-1-107-03534-8},
        note={With a collaboration of S\'{a}ndor Kov\'{a}cs},
}

\bib{k-notQfact}{article}{
      author={Koll\'{a}r, J\'{a}nos},
       title={Relative {MMP} without {$\mathbb{Q}$}-factoriality},
        date={2021},
     journal={Electron. Res. Arch.},
      volume={29},
      number={5},
       pages={3193\ndash 3203},
}

\bib{Kov00}{article}{
      author={Kov\'{a}cs, S\'{a}ndor~J.},
       title={A characterization of rational singularities},
        date={2000},
        ISSN={0012-7094},
     journal={Duke Math. J.},
      volume={102},
      number={2},
       pages={187\ndash 191},
}

\bib{Kov18}{incollection}{
      author={Kov\'{a}cs, S\'{a}ndor~J.},
       title={Non-{C}ohen-{M}acaulay canonical singularities},
        date={2018},
   booktitle={Local and global methods in algebraic geometry},
      series={Contemp. Math.},
      volume={712},
   publisher={Amer. Math. Soc., [Providence], RI},
       pages={251\ndash 259},
}

\bib{KW21}{article}{
      author={Kollár, János},
      author={Witaszek, Jakub},
       title={{Resolution and alteration with ample exceptional divisor}},
        date={2024},
     journal={Sci. China Math., Special Issue on Algebraic Geometry},
}

\bib{LT22}{article}{
      author={Laface, Roberto},
      author={Tirabassi, Sofia},
       title={On ordinary {E}nriques surfaces in positive characteristic},
        date={2022},
     journal={Nagoya Math. J.},
      volume={245},
       pages={192\ndash 205},
}

\bib{LX16}{article}{
      author={Li, C.},
      author={Xu, C.},
       title={Stability of valuations and {K}oll\'{a}r components},
        date={2020},
     journal={J. Eur. Math. Soc. (JEMS)},
      volume={22},
      number={8},
       pages={2573\ndash 2627},
}

\bib{Lyu06}{article}{
      author={Lyubeznik, Gennady},
       title={On the vanishing of local cohomology in characteristic {$p>0$}},
        date={2006},
        ISSN={0010-437X},
     journal={Compos. Math.},
      volume={142},
      number={1},
       pages={207\ndash 221},
}

\bib{Maddock}{article}{
      author={Maddock, Zachary},
       title={Regular del {P}ezzo surfaces with irregularity},
        date={2016},
     journal={J. Algebraic Geom.},
      volume={25},
      number={3},
       pages={401\ndash 429},
}

\bib{Milne(etalecohomology)}{book}{
      author={Milne, James~S.},
       title={\'{E}tale cohomology},
      series={Princeton Mathematical Series},
   publisher={Princeton University Press, Princeton, NJ},
        date={1980},
      volume={No. 33},
}

\bib{Mumford_Abelian_Varieties}{book}{
      author={Mumford, D.},
       title={Abelian varieties},
      series={Tata Institute of Fundamental Research Studies in Mathematics},
   publisher={Published for the Tata Institute of Fundamental Research, Bombay
  by Oxford University Press, London},
        date={1970},
      volume={5},
}

\bib{NT20}{article}{
      author={Nakamura, Yusuke},
      author={Tanaka, Hiromu},
       title={A {W}itt {N}adel vanishing theorem for threefolds},
        date={2020},
     journal={Compos. Math.},
      volume={156},
      number={3},
       pages={435\ndash 475},
}

\bib{Pos24}{article}{
      author={Posva, Quentin},
       title={On the singularities of quotients by 1-foliations},
        date={2026},
        ISSN={0027-7630,2152-6842},
     journal={Nagoya Math. J.},
      volume={261},
       pages={e6},
         url={https://doi.org/10.1017/nmj.2025.10072},
      review={\MR{4999578}},
}

\bib{PZ21}{article}{
      author={Patakfalvi, {Zs}olt},
      author={Zdanowicz, Maciej},
       title={Ordinary varieties with trivial canonical bundle are not
  uniruled},
        date={2021},
     journal={Math. Ann.},
      volume={380},
      number={3-4},
       pages={1767\ndash 1799},
}

\bib{SB97}{article}{
      author={Shepherd-Barron, N.~I.},
       title={Fano threefolds in positive characteristic},
        date={1997},
     journal={Compositio Math.},
      volume={105},
      number={3},
       pages={237\ndash 265},
}

\bib{Sch07}{article}{
      author={Schr\"{o}er, Stefan},
       title={Weak del {P}ezzo surfaces with irregularity},
        date={2007},
     journal={Tohoku Math. J. (2)},
      volume={59},
      number={2},
       pages={293\ndash 322},
}

\bib{stacks-project}{misc}{
      author={{Stacks Project Authors}, The},
       title={\textit{Stacks Project}},
         how={\url{https://stacks.math.columbia.edu}},
}

\bib{Tan18}{article}{
      author={Tanaka, Hiromu},
       title={Minimal model program for excellent surfaces},
        date={2018},
     journal={Ann. Inst. Fourier (Grenoble)},
      volume={68},
      number={1},
       pages={345\ndash 376},
}

\bib{Tan20}{article}{
      author={Tanaka, Hiromu},
       title={Pathologies on {M}ori fibre spaces in positive characteristic},
        date={2020},
     journal={Ann. Sc. Norm. Super. Pisa Cl. Sci. (5)},
      volume={20},
      number={3},
       pages={1113\ndash 1134},
}

\bib{Tan-invariants-imperfect}{article}{
      author={Tanaka, Hiromu},
       title={Invariants of algebraic varieties over imperfect fields},
        date={2021},
     journal={Tohoku Math. J. (2)},
      volume={73},
      number={4},
       pages={471\ndash 538},
}

\bib{tan-p-power}{article}{
      author={Tanaka, Hiromu},
       title={On {$p$}-power freeness in positive characteristic},
        date={2021},
     journal={Math. Nachr.},
      volume={294},
      number={10},
       pages={1968\ndash 1976},
}

\bib{Tan22}{article}{
      author={Tanaka, H.},
       title={Vanishing theorems of {K}odaira type for {W}itt canonical
  sheaves},
        date={2022},
     journal={Selecta Math. (N.S.)},
      volume={28},
      number={1},
       pages={Paper No. 12, 50},
}

\bib{FanoI}{article}{
      author={Tanaka, Hiromu},
       title={Fano threefolds in positive characteristic {I}},
        date={2026},
        ISSN={2156-2261,2154-3321},
     journal={Kyoto J. Math.},
      volume={66},
      number={2},
       pages={427\ndash 498},
         url={https://doi.org/10.1215/21562261-2024-0046},
      review={\MR{5053592}},
}

\bib{Tot19}{article}{
      author={Totaro, Burt},
       title={The failure of {K}odaira vanishing for {F}ano varieties, and
  terminal singularities that are not {C}ohen-{M}acaulay},
        date={2019},
     journal={J. Algebraic Geom.},
      volume={28},
      number={4},
       pages={751\ndash 771},
}

\bib{Tot24}{article}{
      author={Totaro, Burt},
       title={Terminal 3-folds that are not {C}ohen-{M}acaulay},
        date={2024},
     journal={arXiv e-prints},
        note={Available at
  \href{https://arxiv.org/abs/2407.02608}{arXiv:2407.02608}, to appear in
  Transactions of the AMS},
}

\bib{Wal23}{article}{
      author={Waldron, Joe},
       title={Mori fibre spaces for $3$-folds over imperfect fields},
        date={2023},
     journal={arXiv e-prints},
        note={Available at
  \href{https://arxiv.org/abs/2303.00615}{arXiv:2303.00615}},
}

\bib{Wei94}{book}{
      author={Weibel, Charles~A.},
       title={An introduction to homological algebra},
      series={Cambridge Studies in Advanced Mathematics},
   publisher={Cambridge University Press, Cambridge},
        date={1994},
      volume={38},
        ISBN={0-521-43500-5; 0-521-55987-1},
         url={https://doi.org/10.1017/CBO9781139644136},
      review={\MR{1269324}},
}

\bib{XX22}{article}{
      author={Xie, Lingyao},
      author={Xue, Qingyuan},
       title={On the {T}ermination of the {MMP} for {S}emistable {F}ourfolds in
  {M}ixed {C}haracteristic},
        date={2024},
     journal={Michigan Math. J.},
      volume={74},
      number={5},
       pages={977\ndash 999},
}

\bib{Yas19}{article}{
      author={Yasuda, T.},
       title={Discrepancies of {$p$}-cyclic quotient varieties},
        date={2019},
     journal={J. Math. Sci. Univ. Tokyo},
      volume={26},
      number={1},
       pages={1\ndash 14},
}

\end{biblist}
\end{bibdiv}

\end{document}